\documentclass[12pt]{amsart}%
\usepackage{bbding}
\usepackage{dsfont}
\usepackage{graphicx}
\usepackage{amsthm,amscd}
\usepackage{calc}
\usepackage[pagebackref]{hyperref}
\usepackage{mathrsfs}
\usepackage{mathrsfs}
\usepackage{CJK,fancyhdr,amscd}
\usepackage{anysize}
\usepackage{amsmath,amssymb,amsfonts}
\usepackage{epsfig}
\usepackage{latexsym}
\usepackage{hyperref}
\usepackage{indentfirst, latexsym}
\usepackage{graphics}
\usepackage{amsmath}
\usepackage{amsfonts}
\usepackage{amssymb}

\providecommand{\U}[1]{\protect\rule{.1in}{.1in}}

\numberwithin{equation}{section}
\providecommand{\U}[1]{\protect\rule{.1in}{.1in}}
\newtheorem{theorem} {Theorem} [section]
\newtheorem{proposition}[theorem]{Proposition}
\newtheorem{corollary}  [theorem]     {Corollary}
\newtheorem{lemma}  [theorem]     {Lemma}
\newtheorem{example}  [theorem]     {Example}
\newtheorem{remark}  [theorem]     {Remark}
\newtheorem{definition}  [theorem]     {Definition}

\newcommand{\btheorem}{\begin{theorem}}
\newcommand{\etheorem}{\end{theorem}}
\newcommand{\bproposition}{\begin{proposition}}
\newcommand{\eproposition}{\end{proposition}}
\newcommand{\bdefinition}{\begin{definition}}
\newcommand{\edefinition}{\end{definition}}
\newcommand{\bcorollary}{\begin{corollary}}
\newcommand{\ecorollary}{\end{corollary}}
\newcommand{\bproof}{\begin{proof}}
\newcommand{\eproof}{\end{proof}}
\newcommand{\beq}{\begin{equation}}
\newcommand{\eeq}{\end{equation}}
\newcommand{\ee}{\end{eqnarray*}}
\newcommand{\be}{\begin{eqnarray*}}

\newcommand{\elemma}{\end{lemma}}
\newcommand{\blemma}{\begin{lemma}}

\newcommand{\bd}{\begin{enumerate} }
\newcommand{\ed}{\end{enumerate}}

\begin{document}

\title[]{Locally extensions of $\overline\partial$-closed forms on compact generalized Hermitian manifolds}

\author[Kang Wei]{Kang Wei}

\begin{abstract}
In this paper,  we first get a criterion formula for whether a differential form is holomorphic with respect to the generalized complex structure induced by $\epsilon$. Next, we get the  local extensions of $\overline\partial$-closed forms on a smooth family of compact generalized Hermitian manifolds by using this criterion. Finally, as an application, we use this extension to get the invariance of the generalized Hodge number of the deformations of compact
generalized Hermitian manifolds with $\partial\overline\partial$-lemma holds.
\end{abstract}

\renewcommand{\thefootnote}{\fnsymbol{footnote}}

\maketitle

\section{Introduction}

The generalized complex geometry was introduced by N. Hitchin and developed by M.~Gualtieri, G. R.~Cavalcanti and many others in \cite{Hi, Gu1, Gu2}. It is a generalization of both complex geometry and symplectic geometry.
This new geometry provides an indeed broad platform for the people working in both mathematics and physics.
The concept of $H$-twisted was introduced by P. $\check{S}$evera and A. Weinstein in \cite{Se}.
The theory of deformations of complex structures can be dated back to Riemann, and extensively
studied by K. Kodaira, D. C. Spencer, N. Nirenberg, M. Kuranishi and many other great mathematicians in \cite{Ko1,Ko2,Ko3}.
The deformation theory of generalized complex geometry is first studied by M. Gualtieri, R. Goto, Yi Li and so on by using Kodaira-Spencer-Kuranishi's method in \cite{Gu1,Go,Li}.

In this paper, we extend the results of  local extensions of $\overline\partial$-closed forms on a smooth family of compact  Hermitian manifolds in \cite{RaoWanZhao03} to those in generalized case.
 To be precise, if $X$ is a compact generalized complex manifold  with generalized complex structure $J$, we have the $+i$-eigenvalue subspace $L$ in $T_X\oplus T_X^{\ast}$ with respect to the generalized complex structure $J$. We denote $L^{\ast}$ as its duality space with respect to the pair which is defined in Definition \ref{pair} below  and we can identify $L^{\ast} $ with $\overline{L}$ since $<L,L>=0$.
We also know that $\wedge^{\bullet}T_X^{\ast}$ has a Clifford's decomposition
\begin{eqnarray*}
\wedge ^{\bullet}T_{X}^{\ast}&=&U^{-n}(X)\oplus U^{-n+1}(X) \oplus...\oplus U^{n}(X),\end{eqnarray*}
where \begin{eqnarray*}
U^{-n}(X)&:=& \{ \rho \in \wedge ^{\bullet}T_{X}^{\ast}\mid L\cdot \rho =0\},  \\
U^k(X)&:=& \wedge ^{k+n} L^{\ast}\cdot U^{-n}(X)~~  (n=dim_{\mathbb{C}}X),
\end{eqnarray*}
which is defined in Definition \ref{caed2} below.

Let $\pi: \mathcal{X} \to \Delta \subset \mathbb{C}^1 $ be a smooth  family of generalized complex manifolds where $\pi^{-1}(0):=X_0=X$,
$t$ be the local coordinate in the unit disc $\Delta$,
$X_{\epsilon}$ be compact generalized Hermitian manifold with generalized complex structure $J_{\epsilon}$ induced by the generalized Beltrami differentials $\epsilon:=\epsilon(t)$.
We first introduce an isomorphism:

\begin{eqnarray*}
\mathcal{E}:  U^{k-n}(X)    &\to&    U^{k-n}(X_{\epsilon}) ,\\
\sigma  &\mapsto & \mathcal{E}(\sigma)
\end{eqnarray*}
where
\begin{eqnarray*}
\sigma &:=&\frac{1}{k!}\sigma_{i_1\cdots i_k}l^{i_1}\cdot \cdots \cdot l^{i_k}\cdot \rho_0, ~~l^{i_{\alpha}}\in L^{\ast},\rho_0\in U^{-n}(X),\\
 \mathcal{E}(\sigma)&:=& \frac{1}{k!}\sigma_{i_1\cdots i_k} (1+\epsilon^{\ast})(l^{i_1})\cdot \cdots \cdot (1+\epsilon^{\ast})(l^{i_k})\cdot (e^{\epsilon}\cdot \rho_0),\\
  e^{\epsilon}\rho_0&:=& \sum_{i\ge 0} \frac{1}{i!}\epsilon^i \cdot \rho_0.
\end{eqnarray*}

Then we get a criterion formula for whether a differential form is holomorphic with respect to the generalized complex structure $J_{\epsilon}$ induced by $\epsilon$. That is,
\begin{theorem}(=Theorem \ref{k-cri})
Assume that $||\epsilon||_{L^{\infty}}<1$. For any $\sigma \in U^{k-n}(X) \subset \wedge^{\bullet}T_X^{\ast}(0\le k \le 2n),$
\begin{eqnarray*}
\overline\partial_t(\mathcal{E} (\sigma))=0
\Leftrightarrow ([\partial, \epsilon\cdot]+\overline\partial)
\circ (1 - \epsilon\epsilon^{\ast}) ( \sigma)=0,
\end{eqnarray*}
where $\overline\partial_t$ is the $\overline\partial$-operator on $X_{\epsilon}:=X_{\epsilon(t)}$.
\end{theorem}

Next, under the assumption that $X$ satisfies some kind $\partial\overline\partial$-lemma, we get the following local extensions of $\overline\partial$-closed forms on a smooth family $\pi: \mathcal{X} \to \Delta \subset \mathbb{C}^1 $ of compact generalized Hermitian manifolds by using this criterion. The method we used here is parallel to that in \cite{RaoWanZhao03}  which originally came from \cite{Liu02} and developed in \cite{To, Ti, Cl, Ra, Sun01, Sun02, Zh, RaoZhao04, RaoWanZhao02, RaoWanZhao01}.

\begin{theorem}(=Theorem \ref{construction})\label{construction01}
Let $(X,G)$ be a compact Hermitian generalized complex manifold. Assume that $X\in \mathbb{B}^{k-1} \cap \mathbb{S}^{k+1}.$ Then for any $\sigma_{00} \in H_{\overline\partial }^k (X) ~(-n\le k \le n)$,  we can choose

 \begin{eqnarray*}
\sigma_t = \sigma _{00}+\sum_{i,j\ge 1} t^i \bar{t}^j \sigma_{ij} \in U^k(X),
\end{eqnarray*}
such that $\mathcal{E}(\sigma_t) \in U^k(X_{\epsilon})$ and $\overline\partial_t \circ \mathcal{E}(\sigma_t)=0$ with $|t|$ small.

\end{theorem}

Finally, as an application, we use the above local extension formula   to get the invariance of the generalized Hodge number of the deformations of compact
generalized Hermitian manifolds with $\partial\overline\partial$-lemma holds. That is,

\begin{corollary}(=Corollary \ref{i})
Let $(X,G)$ be a compact generalized Hermitian manifold which satisfies $\partial\overline\partial$-lemma, then $h_{\overline\partial_t}^k(X_t)$ is independent of $t$, where $-n\le k \le n$,  $X_t$ is the generalized Hermitian manifold with generalized complex structure $J_{\epsilon}$ induced by $ \epsilon:=\epsilon(t)$.
\end{corollary}

\thanks {\textbf{Acknowledgements:} The author would like to thank Professors Kefeng Liu,   Sheng Rao,  Doctors Jie Tu.}

\section{Preliminaries on compact generalized Hermitian manifolds}

In this section, we review some basic definitions of compact $H$-twisted generalized Hermitian manifolds. We refer the reader to \cite{Hi,Gu1, Gu2, Ba, FG} for details. We also give some propositions which will be used in this paper.

\subsection{Basic knowledges}
We first define a pair on smooth manifold $X$ which is a generalization of the pair $<\frac{\partial}{\partial z^i}, dz^j>=\delta^j_i$.
\begin{definition}\label{pair}
Let $M^{2n}$ be a smooth manifold with $dim_{\mathbb{R}}X=2n$. We define the following pair on $T_X \oplus T_X^{\ast}$:

\begin{eqnarray*}\label{inner product}
<a,b>:=(\xi(Y)+\eta(X)),
\end{eqnarray*}
where $a:=X+\xi, b:=Y+\eta \in T_X \oplus T_X^{\ast},~X, Y \in T_X ,~\xi, \eta \in T_X^{\ast} $.
\end{definition}
\begin{remark}
Comparing with the definition on Page 5 in \cite{Gu1}, we omit the coefficient $\frac{1}{2}$ for convenience.
And since $<a,b>=<b,a>$, we also denote $<a,b>$ by $a(b)$.

\end{remark}

Next, we  introduce the generalized complex structures  on $X$.
\begin{definition}\label{gcssb}
Let $X^{2n}$ be a smooth manifold with $dim_{\mathbb{R}}X=2n$. If there
exists an endomorphism $J$ on $T_X \oplus
T_X^{\ast}$ which satisfies

\begin{displaymath}
J^2=-1,
\end{displaymath}
and
\begin{displaymath}
<a,b>=<Ja,Jb>,~~\mbox{where} ~~a,b \in  T_X \oplus T_X^{\ast},
\end{displaymath}

$J$ is called a generalized almost complex structure on $X$.
\end{definition}

Since $J^2=-1$, we may decompose $T_X \oplus T_X^{\ast}$ into the $\pm i$ -eigenvalue subspaces of $J$:
\begin{displaymath}
T_X \oplus T_X^{\ast}=L+\overline{L},
\end{displaymath}

where $L$  be the $+i$-eigenvalue subspace of $T_X \oplus T_X^{\ast}$.
We define $H$-twisted Courant bracket for some $H \in H^3(M,\mathbb{R})$ as follow:
Let $a=X+\xi ,b=Y+\eta \in T_X \oplus T_X^{\ast}$, where $X,Y \in T_X  ,~\xi, \eta \in T_X^{\ast}$.

\begin{eqnarray*}
[a,b]_H&:=&[X+\xi, Y+\eta]_H   \\
&:=&[X, Y]+L_X\eta -L_Y\xi -\frac{1}{2}d(i_X \eta-i_Y \xi)+i_Y i_X H,
\end{eqnarray*}
where
$
[X,Y]:= XY- YX , L_X\eta :=d \circ i_X \eta+i_X \circ d\eta,~ i_X$ is the contraction by the vector field $X$.
If $J$ is called a generalized almost complex structure on $X$ and $[L,L]_H \subset L, $ we say $J$ is integrable  and call it a generalized complex structure on $X$.

We now introduce Clifford actions on $\wedge ^{\bullet}T_{X}^{\ast}$.
\begin{definition}\label{caed}
Let $\sigma \in \wedge ^{\bullet}T_{X}^{\ast}$ and $a:=X+\xi \in T_X\oplus T_X^{\ast}$, where $X\in T_X,~\xi \in T_X^{\ast}$,
we  define Clifford actions on $\wedge ^{\bullet}T_{X}^{\ast}$:
\begin{displaymath}
a\cdot \sigma:=(X+\xi)\cdot \sigma:=i_X \sigma +\xi \wedge \sigma.
\end{displaymath}

By direct computation, we know that
\begin{eqnarray}
a\cdot b \cdot \sigma +b\cdot a \cdot \sigma &=&a(b) \sigma, \label{wedge}
\end{eqnarray}
where $a,b \in T_X \oplus T_X^{\ast}$.
\end{definition}
If $J$ is integrable, we have that $\overline{L}(\overline{L})=0$ and thus
\begin{eqnarray*}
a\cdot b \cdot \sigma +b\cdot a \cdot \sigma &=&0.
\end{eqnarray*}
Then we get an exterior algebra which denote as $\wedge^{\bullet}\overline{L}$. More generally, we  have that

\begin{eqnarray*}
a\cdot b \cdot \sigma &=& (-1)^{pq}b\cdot a \cdot\sigma,
\end{eqnarray*}
where $a \in \wedge ^{p}\overline{L},~b \in \wedge ^{q}\overline{L}.$
Moreover, since $<L;L>=0$, we can identify $\overline{L}$ with the duality space $L^{\ast}$ of $L$ with respect to the pair $<\cdot ,\cdot >$ in Definition \ref{pair}.
And Courant bracket can be extended into a Schouten bracket on $\wedge
^{\bullet}L^{\ast}=\wedge^{\bullet}\overline{L}$ which we still denote by
$[\cdot,\cdot]_H$ as follows:

\begin{eqnarray*}
[a,b]_H:=\sum_{i,j}[a_i, b_j]_H \cdot a_1\cdot \cdots \cdot \hat{a_i}\cdot \cdots \cdot a_p \cdot b_1\cdot \cdots \cdot \hat{b_j}\cdot \cdots \cdot b_q,
\end{eqnarray*}
where $a=a_1\cdot \cdots \cdot a_p \in \wedge ^{p}L ^{\ast},~b=b_1\cdot \cdots \cdot b_q \in \wedge ^{q}L ^{\ast},~\hat{a_i}$ means omit $a_i$.

Using the Clifford actions, we now introduce the following Clifford's decomposition of $\wedge^{\bullet}T_M^{\ast}$.
\begin{definition}(\cite{Ba}, Page 13)\label{caed2}
Let $(M, J)$ be a  generalized complex manifold  where $J$ is its generalized complex structure. Then we can define

\begin{eqnarray*}
U^{-n}(X)&:=& \{ \rho \in \wedge ^{\bullet}T_{X}^{\ast}\mid L\cdot \rho =0\},  \\
U^k(X)&:=& \wedge ^{k+n} L^{\ast}\cdot U^{-n}(X).
\end{eqnarray*}
\end{definition}

And we know that $U^k(X)$ is the $(-k)i$-eigenvalue subspace of $J$ and $U^{-n}(X)$ is a line bundle which we call it the canonical bundle of $J$.
Thus we can get the Clifford's decomposition of $\wedge ^{\bullet}T_{X}^{\ast}$:
\begin{eqnarray*}
\wedge ^{\bullet}T_{X}^{\ast}=U^{-n}(X)\oplus U^{-n+1}(X) \oplus...\oplus U^{n}(X),
\end{eqnarray*}
where $n=dim_{\mathbb{C}}X$.

We now introduce the twisted de Rham differential $d_H$ for some $H \in H^3(X,\mathbb{R})$ on $\wedge ^{\bullet}T_{X}^{\ast}$ as follows:
\begin{definition}\label{tdd}

\begin{eqnarray*}
d_H: \wedge ^{\bullet}T_{X}^{\ast} &\to& \wedge ^{\bullet}T_{X}^{\ast},\\
\sigma &\mapsto& d\sigma-H\wedge \sigma.
\end{eqnarray*}

We also introduce the twisted Dolbeault operator $\partial_H$ and $\overline\partial_H$ by
\begin{eqnarray*}
\partial_H:=\pi_{k-1}\circ d_H :U^k(X) &\to & U^{k-1}(X), \\
\overline\partial_H:=\pi_{k+1}\circ d_H :U^k(X) &\to & U^{k+1}(X),
\end{eqnarray*}
where $\pi_k$ is the projection onto $U^{k}(X)$.
\end{definition}

We know that $J$ is integrable if and only
if~$d_H=\partial_H+\overline\partial_H$(Page 51 in \cite{Gu1}).

We now introduce the complex $(\wedge ^{\bullet}L ^{\ast},d_L)$ and show its relationship with the complex $(\wedge ^{\bullet}T_{X}^{\ast},\overline\partial_H)$.
\begin{definition}\label{liederivation}
We define the Lie derivation $d_L$ as follows:
\begin{eqnarray*}
d_L: \wedge^k L^{\ast} &\to& \wedge^{k+1} L^{\ast},\\
a &\mapsto &d_L a
\end{eqnarray*}

where
\begin{eqnarray*}
d_L a (x_0, ...x_k)&:=&   \sum_i (-1)^i p(x_i)a (x_0,...,\hat{x_i},...,x_k) \\
&  &+\sum_{i<j}(-1)^{i+j}a ([x_i,x_j]_H,x_0,...,\hat{x_i},...,\hat{x_j},...,x_k),
\end{eqnarray*}
where $a \in \wedge^k L^{\ast}, x_i \in L(0\le i\le k), p:L \to T_X$ is the projection which is called the anchor.
And we  have the following relationship:
\begin{eqnarray*}
\overline\partial_H( a \cdot \rho )= d_L (a) \cdot \rho +(-1)^k a \cdot \overline\partial_H \rho,
\end{eqnarray*}
where $a \in \wedge^k L^{\ast} , \rho \in \wedge^{\bullet} T_X^{\ast},$
\end{definition}

\begin{remark}We now give a discussion about the order of operators $d_H, \partial_H, \overline\partial_H$.
Since $d_H \sigma:=d\sigma -H\wedge\sigma$ , we know that $d_H$ is an operator of order 1. By the definition of $d_L$ above, we know that
$d_L$ is an operator of order 1. For any $\alpha :=a
\cdot \rho \in U^k(X),$ where $a\in \wedge^{k+n} L^{\ast}, \rho\in U^{-n}(X)$, we have that $\partial_H \rho=0$ and thus
$\overline\partial_H( a \cdot \rho )= d_L (a) \cdot \rho +(-1)^k a \cdot \overline\partial_H \rho=d_L (a) \cdot \rho +(-1)^k a \cdot d_H \rho$. So $\overline\partial_H$ is an operator of order 1.  $\partial_H :=d_H-\overline\partial_H$ is also an operator of order 1.
\end{remark}

From now on, we simply denote $d_H, \partial_H, \overline\partial_H, [\cdot , \cdot ]_H$ as $d, \partial, \overline\partial, [\cdot,\cdot]$ respectively.

We list the following lemma which will be used later:
\begin{lemma}[Lemma 2 in Page 8, \cite{Li}, or \cite{Ra2}]\label{braket}
For any $a,b \in \wedge^{2}L^{\ast}, \sigma \in \wedge ^{\bullet}T^{\ast}_{X}$, we have that ~$$[a,b]\cdot \sigma= a\cdot d(b\cdot \sigma)+b\cdot d(a\cdot \sigma)
-a\cdot b \cdot d\sigma -d(a\cdot b \cdot \sigma),$$ where $\cdot$ is denoted as the Clifford action.

\end{lemma}

\subsection{Some elliptic operators}
We first introduce a generalized Hermitian metric on $T_X\oplus T_X^{\ast}$.
\begin{definition}(\cite{Gu2}, Page 3)\label{biip}
Let $(X, J)$ be a generalized complex manifold and $G$ be an endomorphism on $T_X \oplus T_X^{\ast}$ with $G^2=1, GJ=JG$ .  We define
\begin{eqnarray*}
(T_X \oplus T_X^{\ast}) \otimes (T_X \oplus T_X^{\ast})
&\to & C^{\infty }(X),\\
a, b&\mapsto &  <Ga, b>\\
\end{eqnarray*}

where $a, b\in T_X \oplus T_X^{\ast}$ and $<\cdot, \cdot> $ is defined in Definition \ref{pair}. Since $G^2=1, GJ=JG$, $<G\cdot, \cdot>$ is a positive-definite and Hermitian-symmetric metric on $T_X \oplus T_X^{\ast}$. And we call $(X,G)$ a generalized Hermitian manifold.

\end{definition}

The restriction of this metric $<G\cdot, \cdot>$ to the the sub-bundle $T_X$ can be written as a Hermitian metric $\tilde{g}:=g-bg^{-1}b$, where $g$ is a Hermitian metric and $b$ is a 2-form.
And the volume element induced by this metric is

\begin{eqnarray*}
dV_{\tilde{g}} &=&\frac{ \sqrt{ det (g-bg^{-1}b)}}{\sqrt{det (g)}}dV_g\\
&=& \frac {det(g+b)}{det (g)}dV_g.
\end{eqnarray*}

Analogous to the ordinary complex case, we introduce the Hodge $\ast$-operator and use it to define an inner product on $\wedge^{\bullet}T_X^{\ast}$ which we call it Born-Infeld inner products.
Since $G^2=1,$ we have the decomposition $T_X \oplus T_X^{\ast}=C_+ \oplus C_-,$ where $C_{\pm}$ is $\pm 1$-eigenvalue subspace of $G$ respectively. And we choose an orthonormal real basis $\{a_1,\cdots , a_{2n}\}$ of $C_+$ with respect to the metric $<G\cdot, \cdot>.$
We define the Hodge-$\ast$  operator by
\begin{eqnarray*}
\ast &=& a_1\cdot  a_2 \cdot \cdots \cdot a_{2n},
\end{eqnarray*}
which is a product of an oriented orthonormal basis for $C_+$. We have already known that $\ast$ is a real operator, that is, $\ast=\overline{\ast}$.
We now introduce the Born-Infeld inner product $(\cdot, \cdot)$ on $\wedge ^{\bullet}T_{X}^{\ast}$ :
\begin{eqnarray}
(\alpha,\beta)&:=& \int _X \alpha\wedge \sigma (\ast)\bar {\beta},\label{BIinnerproduct}\\
||\alpha||^2&:=&(\alpha,\alpha),\nonumber
\end{eqnarray}
where $\sigma(a_1\cdot \cdots \cdot a_{2n}):=a_{2n}\cdot \cdots \cdot a_1$.

We now introduce some elliptic operators analogous to the complex case.
Let $(X,G)$ be a compact  generalized Hermitian manifold. Denote $d^{\ast}, ~\partial^{\ast},~\overline\partial^{\ast}$ as the adjoint operators of $ d,~ \partial,~\overline\partial$ with respect to the Born-Infeld inner product $( \cdot,\cdot )$ respectively, that is, for any $\alpha, \beta \in \wedge^{\bullet}T_X^{\ast}$,

\begin{eqnarray*}
(\partial \alpha,\beta) &=&  (\alpha,\partial^{\ast}\beta),  \\
(\overline\partial \alpha,\beta) &=& (\alpha,\overline\partial^{\ast}\beta), \\
(d \alpha,\beta) &=& (\alpha,d^{\ast}\beta).
\end{eqnarray*}

\begin{remark}
Now, we give a discussion about the order of $\partial^{\ast}, ~\overline\partial^{\ast}$.
We know that $\partial^{\ast}=\ast \overline\partial \ast^{-1} $ and $\overline\partial^{\ast}=\ast \partial \ast^{-1}$.
Since $\partial, \overline\partial$ are operators of order 1, we have that  $\partial^{\ast}, \overline\partial^{\ast}$ are also operators of order 1

\end{remark}

We also define the Laplacian operators by
\begin{eqnarray*}
\triangle_{d} &:=&  d d^{\ast}+d^{\ast}d, \\
\triangle_{\partial} &:=&  \partial \partial^{\ast}+\partial^{\ast}\partial, \\
\triangle_{\overline\partial} &:=&  \overline\partial \overline\partial^{\ast}+\overline\partial^{\ast}\overline\partial.
\end{eqnarray*}

Then one can show that $~\triangle_{\overline\partial}$ are a self-adjoint operator with respect to inner product $( \cdot,\cdot )$ and
\begin{eqnarray*}
1&= &\mathbb{H} + G\triangle_{\overline\partial}=\mathbb{H} + \triangle_{\overline\partial}G.
\end{eqnarray*}
where $\mathbb{H}$ is the projection onto $ H^{\ast}_{\overline\partial}(X)$, and $G_{\overline\partial}$ is the Green operator corresponding to $\triangle_{\overline\partial}$.
And $\triangle_{d},~\triangle_{\partial}$ have the similar propositions.
If $X$ is a compact  generalized K\"ahler manifold, that is, $J':=-GJ$ is also a generalized complex structure, we  have that
\begin{eqnarray*}
\triangle_{d}& =&  2\triangle_{\partial}=2\triangle_{\overline\partial}.
\end{eqnarray*}

We now introduce  Bott-Chern and Aeppli's Laplacian operators. We refer  the readers  to \cite{Sc, An02, An04, An05} for details.
\begin{definition}

\begin{eqnarray*}
\Delta_{BC}&:=&(\partial\overline\partial)(\partial\overline\partial)^{\ast}+(\partial\overline\partial)^{\ast}(\partial\overline\partial)
+(\overline\partial^{\ast}\partial)(\overline\partial^{\ast}\partial)^{\ast}+(\overline\partial^{\ast}\partial)^{\ast}(\overline\partial^{\ast}\partial)
+\overline\partial^{\ast}\overline\partial+\partial^{\ast}\partial,\\
\Delta_{A}&:=&(\overline\partial\partial)(\overline\partial\partial)^{\ast}+(\overline\partial\partial)^{\ast}(\overline\partial\partial)
+(\partial\overline\partial^{\ast})(\partial\overline\partial^{\ast})^{\ast}+(\partial\overline\partial^{\ast})^{\ast}(\partial\overline\partial^{\ast})
+\overline\partial\overline\partial^{\ast}+\partial\partial^{\ast}.
\end{eqnarray*}

\end{definition}

\begin{lemma}

\begin{eqnarray*}
Ker \Delta_{BC}&=& Ker \partial \cap Ker \overline\partial \cap Ker (\partial\overline\partial)^{\ast},\\
Ker \Delta_{A}&=& Ker \partial^{\ast} \cap Ker \overline\partial^{\ast} \cap Ker (\partial\overline\partial),\\
I&=&\mathbb{H}_{BC}+\Delta_{BC} G_{BC},\\
I&=&\mathbb{H}_{A}+\Delta_{A} G_{A},\\
\wedge^{\bullet}T^{\ast}_X&=& Ker \Delta_{BC} \oplus Im (\partial\overline\partial) \oplus (Im \partial^{\ast} +Im \overline\partial^{\ast}),\\
\wedge^{\bullet}T^{\ast}_X&=& Ker \Delta_{A} \oplus Im (\partial\overline\partial)^{\ast} \oplus (Im \partial +Im \overline\partial).
\end{eqnarray*}
where $\mathbb{H}_{BC}, \mathbb{H}_{A}$ are projections onto $\mathcal{H}^k_{BC}(X):= Ker \Delta_{BC} \cap U^k(X), \mathcal{H}^k_{A}(X):= Ker \Delta_{A} \cap U^k(X),$ and $G_{BC}, G_{A}$ are the Green's operator of $\Delta_{BC}, \Delta_{A}$ respectively.

\end{lemma}

\begin{proof}

For the first formula, since
\begin{eqnarray*}
(\Delta_{BC}\sigma, \sigma)&=& ||(\partial \overline\partial)^{\ast}\sigma||^2+||(\partial \overline\partial)\sigma||^2
+||(\overline\partial^{\ast} \partial)\sigma||^2+||(\overline\partial^{\ast} \partial)^{\ast}\sigma||^2+||\overline\partial \sigma||^2
+||\partial \sigma||^2,
\end{eqnarray*}

we have that $Ker \Delta_{BC}=Ker \partial \cap Ker\overline\partial \cap Ker (\partial\overline\partial)^{\ast}$. The second formula is similar to the
first one.

We now show that  $\Delta_{\overline\partial}, ~\Delta_{BC}$ is strongly elliptic.
Set $L: \{ l_1,\cdots, l_{2n} \}, L^{\ast}: \{ l^1, \cdots , l^{2n} \}$ with $l^j(l_i)=\delta_i^j$ be  a fixed basis of $L$ and $L^{\ast}$ respectively,
Since $[L,L]\subset L,$ we have that
$[l_i,l_j]:=c_{ij}^k l_k,$ where $c_{ij}^k=-c_{ji}^k$.
Then
\begin{eqnarray*}
(d_L l^p) (l_i, l_j)&:=&p(l_i)\delta^p_j-p(l_j)\delta^p_i-c_{ij}^p,
\end{eqnarray*}
that is,
\begin{eqnarray*}
(d_L l^p) &=&\frac{1}{2}(p(l_i)\delta^p_j-p(l_j)\delta^p_i-c_{ij}^p)l^i \cdot l^j.
\end{eqnarray*}

Since we only care about the highest-order terms in the computation, we denote $\approx$ as  the equivalence on the highest-order terms. For example,
\begin{eqnarray*}
(d_L f) &\approx &l_p(d f)l^p,
\end{eqnarray*}
where $f$ is a function.

For any $\sigma \in U^{k}(X), $ we represent $\sigma$ as $\sigma :=f\rho_1  $ locally, where $f$ is a smooth function and $\rho_1:=\frac{1}{(k+n)!}l^{i_1}\cdot \cdots \cdot l^{i_{k+n}}\cdot \rho_0 \in U^{k}(X), ~l^{i_{\alpha}} \in L^{\ast},~ \rho_0\in U^{-n}(X)$. Then we have that

\begin{eqnarray*}
\overline\partial (f\rho_1)
&\approx &d_L (f) \rho_1
\approx (l_p\circ d) (f) l^p \cdot \rho_1 \\
\partial (f\rho_1)
&\approx &(l^i\circ d) (f) \cdot l_i \cdot \rho_1 \\
\overline\partial^{\ast} (f  \rho_1)
&:=& \ast \partial \ast^{-1}(f  \rho_1)\\
&:=&\ast \partial (f  a_{2n}\cdot \cdots \cdot a_1 \cdot\rho_1)\\
&\approx &\ast (l^p\circ d)  (f) \cdot l_p \cdot  a_{2n}\cdot \cdots \cdot a_1 \cdot \rho_1\\
&:=&a_{1}\cdot \cdots \cdot a_{2n} \cdot (l^p\circ d) (f) \cdot l_p \cdot  a_{2n}\cdot \cdots \cdot a_1 \cdot \rho_1)\\
&\approx &(- (l^p\circ d) f <l_p, a^k>a_k+  (l^p\circ d)  (f) \cdot <l_p, b^k>b_k )\cdot \rho_1\\
&=&(l^p\circ d)  (f) \cdot Gl_p\cdot \rho_1
\end{eqnarray*}

The fifth equality holds since the fact that $\ast a \ast^{-1}=-a,  ~\ast b \ast^{-1}=b$ where $a\in C_+, b\in C_-;$
and the sixth equivalence since $Ga=a, Gb=-b$(\cite{FG}).
Thus, we have that
\begin{eqnarray*}
\Delta_{\overline\partial}( f\rho_1)
&=&-(l^i\circ d)(l_p\circ d)(f))\cdot <Gl_i,l^p>\cdot  \rho_1+ \mbox{lower-order terms }.
\end{eqnarray*}

By the fact that $L^{\ast}=\overline{L},$ we change the basis of  $L^{\ast}=\overline{L}$ by $\overline{L}: \{\overline{l_1}, \cdots, \overline{l}_{2n}\}$ such that
$<Gl_i, \overline{l}_j>=\delta_{ij}$ since $G$ is a positive-definite Hermitian metric.
Then the formula above can be written as
\begin{eqnarray*}
\Delta_{\overline\partial}( f\rho_1)&=&-\sum_{i=1}^{2n}(l^i\circ d)(\overline{l}_i\circ d)(f))\cdot  \rho_1+ \mbox{lower-order terms }.
\end{eqnarray*}
and thus $\Delta_{\overline\partial}$ is strongly elliptic.

Now, we compute the highest-order terms of $\Delta_{BC}(f\rho_1).$
Since $JG(L)=GJ(L)=iG(L),$ we have that $G(L) \subset L$ and thus $<GL,L>=0.$
Then we have that
\begin{eqnarray*}
&&l^q\cdot l_j \cdot Gl^p \cdot Gl_i  + Gl^p\cdot l^q \cdot   Gl_i\cdot l_j+  Gl_i \cdot Gl^p \cdot l_j \cdot l^q+Gl_i \cdot l_j \cdot Gl^p \cdot l^q \\
&=&l^q\cdot l_j \cdot Gl^p \cdot Gl_i  - l^q\cdot Gl^p \cdot  l_j \cdot Gl_i+<Gl^p, l^q>Gl_i\cdot l_j-<Gl_i, l_j>Gl^p\cdot l^q\\
&&+  Gl_i \cdot Gl^p \cdot l_j \cdot l^q+Gl_i \cdot l_j \cdot Gl^p \cdot l^q \\
&=&<l^q, Gl_i><l_j , Gl^p>
\end{eqnarray*}

Also, since $(l_i\circ d)(l_j\circ d)=(l_j\circ d)(l_i\circ d)$ , we have that
\begin{eqnarray*}
\Delta_{BC} (f\rho_1)&\approx &(l^i\circ d)(l_p\circ d)(l^j\circ d)(l_q\circ d)(f)\cdot
(l^q\cdot l_j \cdot Gl^p \cdot Gl_i  + Gl^p\cdot l^q \cdot   Gl_i\cdot l_j\\
&&+  Gl_i \cdot Gl^p \cdot l_j \cdot l^q+Gl_i \cdot l_j \cdot Gl^p \cdot l^q )\cdot \rho_1\\
&=&(l^i\circ d)(l_p\circ d)(l^j\circ d)(l_q\circ d)(f)\cdot <Gl_i,l^q><Gl^p,l_j>\cdot  \rho_1.
\end{eqnarray*}

Also, we also change the basis of  $L^{\ast}=\overline{L}$ by $\overline{L}: \{\overline{l_1}, \cdots, \overline{l_{2n}}\}$ such that
$<Gl_i, \overline{l}_j>=\delta_{ij}$. Then the formula above can be written as
\begin{eqnarray*}
\Delta_{BC} (f\rho_1)&=&\sum_{p,q=1}^{2n}(\overline{l}_q\circ d)(l_q\circ d)(\overline{l}_p\circ d)(l_p\circ d)(f)\cdot  \rho_1+\mbox{lower-order terms}.
\end{eqnarray*}
and thus $\Delta_{BC}$ is strongly elliptic.

Now, by the fact that $\Delta_{BC}$ is a strongly elliptic operator, there exists the Green operator $G_{BC}$, and
\begin{eqnarray*}
I=\mathbb{H}_{BC}+\Delta_{BC} G_{BC}.
\end{eqnarray*}

Thus, we have that

\begin{eqnarray*}
\wedge^{\bullet}T^{\ast}_X=Ker \Delta_{BC} +Im (\partial\overline\partial) +Im \partial^{\ast} +Im \overline\partial^{\ast}.
\end{eqnarray*}

For any $x\in Ker \Delta_{BC}, y:=\partial\overline\partial y_1\in  Im (\partial\overline\partial), z:=\partial^{\ast}z_1\in Im \partial^{\ast},
w:=\overline\partial^{\ast} w_1=\in Im \overline\partial^{\ast}$, we have that
\begin{eqnarray*}
(x,y)&:=&(x, \partial\overline\partial y_1)=((\partial\overline\partial) ^{\ast} x,  y_1)=0,\\
(x,z)&:=&(x, \partial^{\ast}z_1)=(\partial x, z_1)=0,\\
(x,w)&:=&(x, \overline\partial^{\ast}w_1)=(\overline\partial x, w_1)=0,\\
(y,z)&:=&(\partial\overline\partial y_1, \partial^{\ast}z_1)=(\partial^2\overline\partial y_1,z_1)=0,\\
(y,w)&:=&(\partial\overline\partial y_1, \overline\partial^{\ast}w_1)=(-\partial\overline\partial^2 y_1,w_1)=0.
\end{eqnarray*}

Thus
\begin{eqnarray*}
\wedge^{\bullet}T^{\ast}_X=Ker \Delta_{BC} \oplus Im (\partial\overline\partial) \oplus (Im \partial^{\ast} +Im \overline\partial^{\ast}).
\end{eqnarray*}

Similarly, we have that $\Delta_{A}$ is also strongly elliptic and has the above results.

\end{proof}

\begin{remark}
If $X$ is a compact generalized K\"ahler manifold, we have that
\begin{eqnarray*}
\overline\partial^{\ast}\partial&=& -\partial \overline\partial^{\ast}, ~\overline\partial \partial^{\ast}=-
\partial^{\ast}\overline\partial,\\
\overline\partial^{\ast}\overline\partial&=&-\overline\partial\overline\partial^{\ast}, ~\partial \partial^{\ast}
=-\partial^{\ast}\partial.
\end{eqnarray*}

Then

\begin{eqnarray*}
\Delta_{BC}&=&\Delta^2_{\overline\partial}+\overline\partial ^{\ast}\overline\partial +\partial ^{\ast} \partial.
\end{eqnarray*}

So, by the fact that  $\Delta_{\overline\partial}$ is elliptic,  $\Delta_{BC}$ is also elliptic.

\end{remark}

For the Bott-Chern and Aeppli's Laplacian operators, we have the following lemma which  will be used later.
\begin{lemma}\label{laplacianbca}

\begin{eqnarray*}
\Delta _{BC}(\partial\overline\partial)(\partial\overline\partial)^{\ast}&=&(\partial\overline\partial)(\partial\overline\partial)^{\ast}\Delta _{BC},\\
\Delta _{A}(\overline\partial\partial)^{\ast}(\overline\partial\partial)&=&(\overline\partial\partial)^{\ast}(\overline\partial\partial)\Delta _{A},\\
\Delta_{BC}(\partial \overline\partial)&=&(\partial \overline\partial)(\partial \overline\partial)^{\ast}(\partial \overline\partial)
=(\partial \overline\partial)\Delta_{A},\\
(\partial\overline\partial)^{\ast}\Delta_{BC}&=& (\partial\overline\partial)^{\ast}(\partial\overline\partial)(\partial\overline\partial)^{\ast}
=\Delta_{A}(\partial\overline\partial)^{\ast},\\
G _{BC}(\partial\overline\partial)(\partial\overline\partial)^{\ast}&=&(\partial\overline\partial)(\partial\overline\partial)^{\ast}G _{BC},\\
G _{A}(\overline\partial\partial)^{\ast}(\overline\partial\partial)&=&(\overline\partial\partial)^{\ast}(\overline\partial\partial)G _{A},\\
G_{BC}(\partial \overline\partial)&=&(\partial \overline\partial)G_{A},\\
(\partial\overline\partial)^{\ast}G_{BC}&=& G_{A}(\partial\overline\partial)^{\ast}.
\end{eqnarray*}

\end{lemma}

\begin{proof}

By definition of $\Delta_{BC}, ~\Delta_A$, we can get the first four formula. For example, we have that
\begin{eqnarray*}
\Delta _{BC}(\partial\overline\partial)(\partial\overline\partial)^{\ast}
&=&(\partial\overline\partial)(\partial\overline\partial)^{\ast}(\partial\overline\partial)(\partial\overline\partial)^{\ast}
=(\partial\overline\partial)(\partial\overline\partial)^{\ast}\Delta _{BC}.
\end{eqnarray*}

By the first formula,  we have that

\begin{eqnarray*}
G_{BC}\Delta _{BC}(\partial\overline\partial)(\partial\overline\partial)^{\ast}G_{BC}&=&G_{BC}(\partial\overline\partial)(\partial\overline\partial)^{\ast}\Delta _{BC}G_{BC},\\
=(1-\mathbb{H}_{BC})(\partial\overline\partial)(\partial\overline\partial)^{\ast}G_{BC}&=&G_{BC}(\partial\overline\partial)
(\partial\overline\partial)^{\ast}(I-\mathbb{H}_{BC}),\\
=(\partial\overline\partial)(\partial\overline\partial)^{\ast}G_{BC}&=&G_{BC}(\partial\overline\partial)(\partial\overline\partial)^{\ast}.
\end{eqnarray*}

So we prove the fifth formula and the  proofs of the rest formulas are similar.

\end{proof}

\subsection{Deformations}In this subsection, we introduce some propositions about deformations of compact generalized Hermitian manifolds.
Let $\pi: \mathcal{X} \to \Delta \subset \mathbb{C}^1 $ be a smooth  family of generalized complex manifolds where $\pi^{-1}(0):=X_0:=X$ is a compact  generalized Hermitian manifold,  $t$ be the local holomorphic coordinate on the unit disc $\Delta\subset \mathbb{C}^1$, $\epsilon(t)$  be a generalized Beltrami differentials from Kodaira-Spencer-Kuranishi's theory, and $M_{\epsilon(t)}$ be the generalized
K\"ahler manifold with generalized complex structure induced by $\epsilon(t)$.
We have already known that ~$$T_{X}\oplus T^{\ast}_{X}=L \oplus L^{\ast},$$ where $L$ is the $+i$-eigenvalue subspace of the generalized complex structure $J$ on $X$.
We assume that ~$$L: \{ l_1, l_2, \cdots , l_{2n}\},  ~~L^{\ast}: \{ l^1, l^2, \cdots , l^{2n}\}$$ are basis of $L$ and $ L^{\ast}$ respectively with $l^i(l_j)=\delta ^i_j$.

Since the generalized Beltrami differentials $\epsilon \in \wedge^{2}L^{\ast}$, we locally  have ~$$\epsilon=\frac{1}{2}\epsilon_{ij}l^i \cdot l^j,$$ with $l^i \cdot l^j=-l^i \cdot l^j, \epsilon_{ij}=-\epsilon_{ji}$. Then

\begin{eqnarray*}
\epsilon (l_p)&=&\frac{1}{2}\epsilon_{ij}l^i \cdot l^j(l_p)\\
&=&\frac{1}{2}\epsilon_{ip}l^i-\frac{1}{2}\epsilon_{pj}l^j\\
&=&\epsilon_{ip}l^i \in L^{\ast}.
\end{eqnarray*}

Let $(X_{\epsilon}, J_{\epsilon})$ be the  generalized complex structure induced by $\epsilon$. We have known that the $+i$-eigenvalue subspace in $T_X\oplus T_X^{\ast}$ corresponding to $J_{\epsilon}$ is $L_{\epsilon}=(1+\epsilon)( L).$ We assume the basis ~$$L_{\epsilon}:\{ \xi_1, \xi_2, \cdots, \xi_{2n}\},~~L^{\ast}_{\epsilon}:\{ \xi^1, \xi^2, \cdots, \xi^{2n}\}$$ of $L_{\epsilon}$ and $L^{\ast}_{\epsilon}$ respectively with $\xi^i(\xi_j)=\delta^i_j$. Then we have that

\begin{eqnarray}\label{deformation}
l_j (\xi_i)=\epsilon_{jk}l^k (\xi_i).
\end{eqnarray}

In matrix form , that is,
\begin{eqnarray*}
\left(\begin{array}{ccc}
l_1(\xi_1) &  \cdots &  l_{2n}(\xi_1)\\
\cdots &  \cdots  &   \cdots\\
l_1(\xi_{2n}) &  \cdots &  l_{2n}(\xi_{2n})
\end{array}\right)
&=&
\left(\begin{array}{ccc}
l^1(\xi_1) &  \cdots &  l^{2n}(\xi_1)\\
\cdots &  \cdots  &   \cdots\\
l^1(\xi_{2n}) &  \cdots &  l^{2n}(\xi_{2n})
\end{array}\right)
\left(\begin{array}{ccc}
\epsilon_{11} &  \cdots &  \epsilon_{2n,1}\\
\cdots &  \cdots  &   \cdots\\
\epsilon_{1,2n} &  \cdots &  \epsilon_{2n,2n},
\end{array}\right)
\end{eqnarray*}
and we simply denote it as
~$$L(L_{\epsilon})=L^{\ast}(L_{\epsilon}) [\epsilon],$$ or equivalently,
\begin{eqnarray}\label{deformationmatrixform}
L^{\ast}(L_{\epsilon})^{-1} L(L_{\epsilon})= [\epsilon].
\end{eqnarray}

We also define  $\epsilon^{\ast}:=\frac{1}{2}\epsilon^{ij}l_i\cdot l_j \in \wedge ^2 L$  where  $\epsilon^{ij}=-\epsilon^{ji},~ l_i\cdot l_j=-l_j\cdot l_i$  such that  $L_{\epsilon}^{\ast}=(1+\epsilon^{\ast})( L^{\ast})$.
In matrix form, we have that
\begin{eqnarray*}
\left(\begin{array}{ccc}
l^1(\xi^1) &  \cdots &  l^{2n}(\xi^1)\\
\cdots &  \cdots  &   \cdots\\
l^1(\xi^{2n}) &  \cdots &  l^{2n}(\xi^{2n})
\end{array}\right)
&=&
\left(\begin{array}{ccc}
l_1(\xi^1) &  \cdots &  l_{2n}(\xi^1)\\
\cdots &  \cdots  &   \cdots\\
l_1(\xi^{2n}) &  \cdots &  l_{2n}(\xi^{2n})
\end{array}\right)
\left(\begin{array}{ccc}
\epsilon^{11} &  \cdots &  \epsilon^{2n,1}\\
\cdots &  \cdots  &   \cdots\\
\epsilon^{1,2n} &  \cdots &  \epsilon^{2n,2n},
\end{array}\right)
\end{eqnarray*}
and we simply denote it as
~$$L^{\ast}(L_{\epsilon}^{\ast})=L(L_{\epsilon}^{\ast})[\epsilon^{\ast}],$$
or equivalently,
\begin{eqnarray}\label{deformationmatrixform02}
L(L_{\epsilon}^{\ast})^{-1} L^{\ast}(L_{\epsilon}^{\ast})= [\epsilon^{\ast}].
\end{eqnarray}

\begin{remark}
If $X$ is a compact complex manifold, we have that $L=T^{\ast 1,0}_X \oplus T^{0,1}_X,  L^{\ast}=T^{ 1,0}_X \oplus T^{\ast 0,1}_X$.
The Beltrami differentials $\varphi= \varphi_{\bar{j}}^i dz ^{\bar{j}} \frac{\partial}{\partial z^i}\in  A^{0,1}(X, T_X^{,0})$
and $\varphi (d z ^{\bar{p}})=0,  \varphi (d z ^{p})=\varphi_{\bar{j}}^p d z ^{\bar{j}}\in L^{\ast}$.

If we choose the basis ~$$T_X ^{\ast 1,0}:\{dz^i\}, T_{X_{\varphi}} ^{\ast 1,0}:\{d\zeta^i\},$$
we have that

\begin{eqnarray*}
\frac{\partial \zeta ^i}{\partial z ^{\bar{j}}} =\varphi ^k_{\bar{j}} \frac{\partial \zeta ^i}{\partial z ^{k}}.
\end{eqnarray*}

That is,
\begin{eqnarray*}
\left(\begin{array}{ccc}
\frac{\partial \zeta ^1}{\partial z ^{\bar{1}}} &  \cdots &  \frac{\partial \zeta ^1}{\partial z ^{\bar{n}}}\\
\cdots &  \cdots  &   \cdots\\
\frac{\partial \zeta ^n}{\partial z ^{\bar{1}}} &  \cdots &  \frac{\partial \zeta ^n}{\partial z ^{\bar{n}}}
\end{array}\right)
=
\left(\begin{array}{ccc}
\frac{\partial \zeta ^1}{\partial z ^{1} }&  \cdots &  \frac{\partial \zeta ^1}{\partial z ^{n}}\\
\cdots &  \cdots  &   \cdots\\
\frac{\partial \zeta ^n}{\partial z ^{1}} &  \cdots &  \frac{\partial \zeta ^n}{\partial z ^{n}}
\end{array}\right)
\left(\begin{array}{ccc}
\varphi^1_{\bar{1}} &  \cdots &  \varphi^n_{\bar{1}}\\
\cdots &  \cdots  &   \cdots\\
\varphi^n_{\bar{1}}&  \cdots &  \varphi^n_{\bar{n}}
\end{array}\right)
\end{eqnarray*}
and we simply denote it as
~$
(\zeta)_{\bar{z}}=\zeta_z [\varphi]$,   or equivalently, $[\varphi] =\zeta _z ^{-1}(\zeta)_{\bar{z}}.$

\end{remark}

The following propositions show us the relationship between the $+i$-eigenvalue subspace $L$ and $L_{\epsilon}$ induced by the generalized complex structure
$J:=J_0$ and $J_{\epsilon}$ respectively.

\begin{proposition}\label{unholpart}
\begin{eqnarray*}
\left(\begin{array}{cc}
L(L_{\epsilon}^{\ast}) &   L^{\ast}(L_{\epsilon}^{\ast})\\
L(L_{\epsilon}) &  L^{\ast}(L_{\epsilon})\end{array}\right)^{-1}
&=&\left(\begin{array}{cc}
L_{\epsilon}(L^{\ast}) &   L_{\epsilon}^{\ast}(L^{\ast})\\
L_{\epsilon}(L) &  L_{\epsilon}^{\ast}(L)\end{array}\right)\\
&=&
\left(\begin{array}{cc}
 (1-[\epsilon^{\ast}][\epsilon])^{-1}L(L_{\epsilon}^{\ast})^{-1}
&  -[\epsilon^{\ast}] (1-[\epsilon][\epsilon^{\ast}])^{-1}L^{\ast}(L_{\epsilon})^{-1}\\
-(1-[\epsilon] [\epsilon^{\ast}])^{-1}[\epsilon] L(L_{\epsilon}^{\ast})^{-1} & (1-[\epsilon][\epsilon^{\ast}])^{-1}L^{\ast}(L_{\epsilon})^{-1}
\end{array}\right).
\end{eqnarray*}

Here $[\epsilon], [\epsilon^{\ast}]$ means the matrix forms of $\epsilon, \epsilon^{\ast}$ respectively as those in  Formula \ref{deformationmatrixform} and Formula \ref{deformationmatrixform02}.

\end{proposition}

\begin{proof}
Since
\begin{eqnarray*}
\xi_i=l^k(\xi_i)l_k+l_k(\xi_i)l^k,
\end{eqnarray*}
we have that
\begin{eqnarray*}
\delta_i^j&=&\xi^j(\xi_i)=l^k(\xi_i)\xi^j(l_k)+l_k(\xi_i)\xi^j(l^k),\\
0&=&\xi_j(\xi_i)=l^k(\xi_i)\xi_j(l_k)+l_k(\xi_i)\xi_j(l^k).
\end{eqnarray*}
In matrix form, that is,
\begin{eqnarray*}
I_{2n\times 2n}&=&L^{\ast}(L_{\epsilon})L^{\ast}_{\epsilon}(L)+L(L_{\epsilon})L^{\ast}_{\epsilon}(L^{\ast}),\\
0&=&L^{\ast}(L_{\epsilon})L_{\epsilon}(L)+L(L_{\epsilon})L_{\epsilon}(L^{\ast}).
\end{eqnarray*}

Similarly, since
\begin{eqnarray*}
\xi^i=l^k(\xi^i)l_k+l_k(\xi^i)l^k,
\end{eqnarray*}
we have that
\begin{eqnarray*}
0&=&\xi^j(\xi^i)=l^k(\xi^i)\xi^j(l_k)+l_k(\xi^i)\xi^j(l^k),\\
\delta_j^i&=&\xi_j(\xi^i)=l^k(\xi^i)\xi_j(l_k)+l_k(\xi^i)\xi_j(l^k).
\end{eqnarray*}
In matrix form, that is,
\begin{eqnarray*}
0&=&L^{\ast}(L_{\epsilon}^{\ast})L^{\ast}_{\epsilon}(L)+L(L_{\epsilon}^{\ast})L^{\ast}_{\epsilon}(L^{\ast}),\\
I_{2n\times 2n}&=&L^{\ast}(L_{\epsilon}^{\ast})L_{\epsilon}(L)+L(L_{\epsilon}^{\ast})L_{\epsilon}(L^{\ast}).
\end{eqnarray*}

Since
\begin{eqnarray*}
l_i=\xi^k(l_i)\xi_k+\xi_k(l_i)\xi^k,
\end{eqnarray*}
we have that
\begin{eqnarray*}
\delta_i^j&=&l^j(l_i)=\xi^k(l_i)l^j(\xi_k)+\xi_k(l_i)l^j(\xi^k),\\
0&=&l_j(l_i)=\xi^k(l_i)l_j(\xi_k)+\xi_k(l_i)l_j(\xi^k).
\end{eqnarray*}
In matrix form, that is,
\begin{eqnarray*}
I_{2n\times 2n}&=&L^{\ast}_{\epsilon}(L)L^{\ast}(L_{\epsilon})+L_{\epsilon}(L)L^{\ast}(L^{\ast}_{\epsilon}),\\
0&=&L^{\ast}_{\epsilon}(L)L(L_{\epsilon})+L_{\epsilon}(L)L(L^{\ast}_{\epsilon}).
\end{eqnarray*}

And since
\begin{eqnarray*}
l^i=\xi^k(l^i)\xi_k+\xi_k(l^i)\xi^k,
\end{eqnarray*}
we have that
\begin{eqnarray*}
\delta^i_j&=&l_j(l^i)=\xi^k(l^i)l_j(\xi_k)+\xi_k(l^i)l_j(\xi^k),\\
0&=&l^j(l^i)=\xi^k(l^i)l^j(\xi_k)+\xi_k(l^i)l^j(\xi^k).
\end{eqnarray*}
In matrix form, that is,
\begin{eqnarray*}
I_{2n\times 2n}&=&L^{\ast}_{\epsilon}(L^{\ast})L(L_{\epsilon})+L_{\epsilon}(L^{\ast})L(L^{\ast}_{\epsilon}),\\
0&=&L^{\ast}_{\epsilon}(L^{\ast})L^{\ast}(L_{\epsilon})+L_{\epsilon}(L^{\ast})L^{\ast}(L^{\ast}_{\epsilon}).
\end{eqnarray*}

Combine the above, we have that
\begin{eqnarray*}
\left(\begin{array}{cc}
L(L_{\epsilon}^{\ast})&  L^{\ast}(L_{\epsilon}^{\ast})\\
L(L_{\epsilon}) &  L^{\ast}(L_{\epsilon})
\end{array}\right)
\left(\begin{array}{cc}
L_{\epsilon}(L^{\ast}) &  L_{\epsilon}^{\ast}(L^{\ast})\\
L_{\epsilon}(L)  &   L_{\epsilon}^{\ast}(L)
\end{array}\right)
&=&I_{4n \times 4n},\\
\left(\begin{array}{cc}
L_{\epsilon}(L^{\ast}) &  L_{\epsilon}^{\ast}(L^{\ast})\\
L_{\epsilon}(L)  &   L_{\epsilon}^{\ast}(L)
\end{array}\right)
\left(\begin{array}{cc}
L(L_{\epsilon}^{\ast})&  L^{\ast}(L_{\epsilon}^{\ast})\\
L(L_{\epsilon}) &  L^{\ast}(L_{\epsilon})
\end{array}\right)
&=&I_{4n \times 4n}.
\end{eqnarray*}

Thus,
\begin{eqnarray*}
\left(\begin{array}{cc}
L(L_{\epsilon}^{\ast})&  L^{\ast}(L_{\epsilon}^{\ast})\\
L(L_{\epsilon}) &  L^{\ast}(L_{\epsilon})
\end{array}\right)^{-1}&=&
\left(\begin{array}{cc}
L_{\epsilon}(L^{\ast}) &  L_{\epsilon}^{\ast}(L^{\ast})\\
L_{\epsilon}(L)  &   L_{\epsilon}^{\ast}(L)
\end{array}\right).
\end{eqnarray*}

Since

\begin{eqnarray*}
&&\left(\begin{array}{cc}
I_{2n \times 2n}                           &    0\\
-L(L_{\epsilon})L(L_{\epsilon}^{\ast})^{-1}&  I_{2n \times 2n})
\end{array}\right)
\left(\begin{array}{cc}
L(L_{\epsilon}^{\ast})&  L^{\ast}(L_{\epsilon}^{\ast})\\
L(L_{\epsilon}) &  L^{\ast}(L_{\epsilon})
\end{array}\right)\\
&=&
\left(\begin{array}{cc}
L(L_{\epsilon}^{\ast})&  L^{\ast}(L_{\epsilon}^{\ast})\\
0 &  L^{\ast}(L_{\epsilon})-L(L_{\epsilon})L(L_{\epsilon}^{\ast})^{-1}L^{\ast}(L_{\epsilon}^{\ast})
\end{array}\right),
\end{eqnarray*}

by using the following lemma:
\begin{lemma}\label{inversematrix}
\begin{eqnarray*}
\left(\begin{array}{cc}
A &  C\\
0  &   B
\end{array}\right)^{-1}
=
\left(\begin{array}{cc}
A^{-1} &  -A^{-1}CB^{-1}\\
0  &   B^{-1}
\end{array}\right)
\end{eqnarray*}
\end{lemma}

 and Formula \ref{deformationmatrixform}, \ref{deformationmatrixform02}, we  get that
\begin{eqnarray*}
&&\left(\begin{array}{cc}
L_{\epsilon}(L^{\ast}) &  L_{\epsilon}^{\ast}(L^{\ast})\\
L_{\epsilon}(L)  &   L_{\epsilon}^{\ast}(L)
\end{array}\right)\\
&=&
\left(\begin{array}{cc}
L(L_{\epsilon}^{\ast})&  L^{\ast}(L_{\epsilon}^{\ast})\\
0 &  L^{\ast}(L_{\epsilon})-L(L_{\epsilon})L(L_{\epsilon}^{\ast})^{-1}L^{\ast}(L_{\epsilon}^{\ast})
\end{array}\right)^{-1}
\left(\begin{array}{cc}
I_{2n \times 2n}                           &    0\\
-L(L_{\epsilon})L(L_{\epsilon}^{\ast})^{-1}&  I_{2n \times 2n}
\end{array}\right)\\
&=&
\left(\begin{array}{cc}
L(L_{\epsilon}^{\ast})^{-1}   &  -L(L_{\epsilon}^{\ast})^{-1}  L^{\ast}(L_{\epsilon}^{\ast})  (1-[\epsilon][\epsilon^{\ast}])^{-1}L^{\ast}(L_{\epsilon})^{-1}\\
0 & (1-[\epsilon][\epsilon^{\ast}])^{-1}L^{\ast}(L_{\epsilon})^{-1}
\end{array}\right)
\left(\begin{array}{cc}
I_{2n \times 2n}                           &    0\\
-L(L_{\epsilon})L(L_{\epsilon}^{\ast})^{-1}&  I_{2n \times 2n}
\end{array}\right)\\
&=&
\left(\begin{array}{cc}
L(L_{\epsilon}^{\ast})^{-1}   &  -[\epsilon^{\ast}] (1-[\epsilon ][\epsilon^{\ast}])^{-1}L^{\ast}(L_{\epsilon})^{-1}\\
0 & (1-[\epsilon][ \epsilon^{\ast}])^{-1}L^{\ast}(L_{\epsilon})^{-1}
\end{array}\right)
\left(\begin{array}{cc}
I_{2n \times 2n}                           &    0\\
-L(L_{\epsilon})L(L_{\epsilon}^{\ast})^{-1}&  I_{2n \times 2n}
\end{array}\right)\\
&=&
\left(\begin{array}{cc}
 (1+  [\epsilon^{\ast}] (1-[\epsilon][\epsilon^{\ast}])^{-1}L^{\ast}(L_{\epsilon})^{-1}L(L_{\epsilon}))L(L_{\epsilon}^{\ast})^{-1}
&  -[\epsilon^{\ast}] (1-[\epsilon][\epsilon^{\ast}])^{-1}L^{\ast}(L_{\epsilon})^{-1}\\
-(1-[\epsilon][\epsilon^{\ast}])^{-1}L^{\ast}(L_{\epsilon})^{-1}L(L_{\epsilon})L(L_{\epsilon}^{\ast})^{-1} & (1-[\epsilon][\epsilon^{\ast}])^{-1}L^{\ast}(L_{\epsilon})^{-1}
\end{array}\right)\\
&=&
\left(\begin{array}{cc}
 (1-[\epsilon^{\ast}][\epsilon])^{-1}L(L_{\epsilon}^{\ast})^{-1}
&  -[\epsilon^{\ast}] (1-[\epsilon][\epsilon^{\ast}])^{-1}L^{\ast}(L_{\epsilon})^{-1}\\
-(1-[\epsilon] [\epsilon^{\ast}])^{-1}[\epsilon] L(L_{\epsilon}^{\ast})^{-1} & (1-[\epsilon][\epsilon^{\ast}])^{-1}L^{\ast}(L_{\epsilon})^{-1}
\end{array}\right).
\end{eqnarray*}

The last equivalence since
\begin{eqnarray*}
&&1+  [\epsilon^{\ast}] (1-[\epsilon][ \epsilon^{\ast}])^{-1}L^{\ast}(L_{\epsilon})^{-1}L(L_{\epsilon})\\
&=&1+  [\epsilon^{\ast}] (1-[\epsilon][ \epsilon^{\ast}])^{-1}[\epsilon]
=(1-[\epsilon^{\ast}][\epsilon])^{-1}\\
\Leftrightarrow  1
&=&(1+  [\epsilon^{\ast}] (1-[\epsilon][ \epsilon^{\ast}])^{-1}[\epsilon])
(1-[\epsilon^{\ast}][\epsilon])\\
&=&1+  [\epsilon^{\ast}] (1-[\epsilon][ \epsilon^{\ast}])^{-1}[\epsilon]-[\epsilon^{\ast}][\epsilon]- [\epsilon^{\ast}] (1-[\epsilon ][ \epsilon^{\ast}])^{-1}[\epsilon][\epsilon^{\ast}][\epsilon]\\
\Leftrightarrow   0&=&[\epsilon^{\ast}] (1-[\epsilon][ \epsilon^{\ast}])^{-1}[\epsilon]-[\epsilon^{\ast}][\epsilon]-[\epsilon^{\ast}] (1-[\epsilon][ \epsilon^{\ast}])^{-1}[\epsilon][\epsilon^{\ast}][\epsilon]\\
\Leftrightarrow  1&=& (1-[\epsilon][ \epsilon^{\ast}])^{-1}- (1-[\epsilon][ \epsilon^{\ast}])^{-1}[\epsilon][\epsilon^{\ast}]\\
\Leftrightarrow 1&=& (1-[\epsilon][ \epsilon^{\ast}])^{-1}(1- [\epsilon][\epsilon^{\ast}]).
\end{eqnarray*}

\end{proof}

\begin{remark}

\begin{eqnarray*}
(1-[\epsilon][\epsilon^{\ast}])^{-1}[\epsilon] =[\epsilon] (1-[\epsilon^{\ast}][\epsilon])^{-1}
\end{eqnarray*}

always holds since it is equivalent to
\begin{eqnarray*}
(1-[\epsilon][\epsilon^{\ast}]) [\epsilon]=[\epsilon] (1-[\epsilon^{\ast}][\epsilon])=[\epsilon]-[\epsilon][\epsilon^{\ast}][\epsilon].
\end{eqnarray*}

\end{remark}

\begin{example}
If $X$ is a compact complex manifold and we assume that $dim_{\mathbb{C}}X=1$ for convenience. Then we have the basis
\begin{eqnarray*}
L: \{dz, \frac{\partial}{\partial \bar{z}}\},
L_{\epsilon}: \{d\zeta, \frac{\partial}{\partial \bar{\zeta}}  \};\\
L: \{ \frac{\partial}{\partial z}, d\bar{z}\},
L_{\epsilon}: \{ \frac{\partial}{\partial \zeta},  d\bar{\zeta}\}.
\end{eqnarray*}

Then
\begin{eqnarray*}
d\zeta &=&  \zeta_z dz + (\zeta)_{\bar{z}}d \bar{z}\\
&=&\frac{\partial}{\partial z}( d\zeta) dz + \frac{\partial}{\partial \bar{z}}(d\zeta) d \bar{z},\\
\frac{\partial}{\partial \zeta} &=& z_{\zeta}\frac{\partial}{\partial z}+  (\bar{z})_{ \zeta} \frac{\partial}{\partial \bar{z}}\\
&=& dz(\frac{\partial}{\partial \zeta})\frac{\partial}{\partial z}+  d\bar{z} (\frac{\partial}{\partial \zeta}) \frac{\partial}{\partial \bar{z}}.
\end{eqnarray*}

Thus
\begin{eqnarray*}
\zeta_z=\frac{\partial}{\partial z}( d\zeta),  (\zeta)_{\bar{z}}=\frac{\partial}{\partial \bar{z}}(d\zeta),\\
z_{\zeta}=dz(\frac{\partial}{\partial \zeta}),  (\bar{z})_{ \zeta}=d\bar{z} (\frac{\partial}{\partial \zeta}).
\end{eqnarray*}

\begin{eqnarray*}
\left(\begin{array}{cc}
L(L_{\epsilon}^{\ast})&  L^{\ast}(L_{\epsilon}^{\ast})\\
L(L_{\epsilon}) &  L^{\ast}(L_{\epsilon})
\end{array}\right)^{-1}&=&
\left(\begin{array}{cccc}
    z_{\zeta}      &          0                  &        0                 &     (\bar{z})_{\zeta}\\
0                  &  (\bar{\zeta})_{\bar{z}}   &      (\bar{\zeta})_{z}     &     0\\
0                  &   (\zeta)_{\bar{z}}       &           \zeta _{z}         &     0\\
(z)_{\bar{\zeta}} &             0             &           0                  &       (\bar{z})_{\bar{\zeta}}
\end{array}\right)^{-1}\\
&=&
\left(\begin{array}{cccc}
   \zeta_{z}       &          0                  &        0                 &     (\bar{\zeta})_{z} \\
0                  &   (\bar{z})_{\bar{\zeta}}  &       (\bar{z})_{\zeta}   &     0\\
0                  &  (z)_{\bar{\zeta}}       &         z _{\zeta}          &     0\\
  (\zeta)_{\bar{z}}&             0             &           0                  &       (\bar{\zeta})_{\bar{z}}
\end{array}\right)\\
&=&
\left(\begin{array}{cc}
L_{\epsilon}(L^{\ast}) &  L_{\epsilon}^{\ast}(L^{\ast})\\
L_{\epsilon}(L)  &   L_{\epsilon}^{\ast}(L)
\end{array}\right).
\end{eqnarray*}
\end{example}

\begin{proposition}\label{001}
\begin{eqnarray*}
\xi_i=l^q(\xi_i)(1+\epsilon)( l_q),\\
\xi^i=l_q(\xi^i)(1+\epsilon^{\ast})( l^q).
\end{eqnarray*}
\end{proposition}

\begin{proof}
By Formula \ref{deformation}, we have that
\begin{eqnarray*}
\xi_i&=&l_p(\xi_i)l^p+l^q(\xi_i)l_q\\
&=&\epsilon_{pq}l^q(\xi_i)l^p+l^q(\xi_i)l_q\\
&=&l^q(\xi_i)(\epsilon_{pq}l^p+l_q)\\
&=&l^q(\xi_i)(1+\epsilon)(l_q).
\end{eqnarray*}
And the second formula is the duality of the first one.

\end{proof}

Finally , we list the following lemma which will be used later.
\begin{lemma}[\cite{KW}]\label{holcri}
Let $\epsilon \in \wedge^{2}L^{\ast}$ be an  integrable generalized Beltrami differentials, that is, $d_{L}\epsilon =\frac{1}{2}[\epsilon, \epsilon].$ For any $\sigma \in \wedge ^{\bullet}T^{\ast}_{X}$, we have that
~$$e^{-\epsilon}\cdot d \circ e^{\epsilon} \cdot \sigma=( d + [\partial , \epsilon\cdot]) \circ\sigma,$$
where $[\partial , \epsilon\cdot]:=\partial\circ \epsilon \cdot - \epsilon \cdot \partial \circ.$

\end{lemma}

\subsection{$\partial \overline\partial $-Lemma} In this subsection, we discuss some propositions about $\partial \overline\partial $-lemma on $X$.

\begin{definition}[$\partial \overline\partial $-Lemma]
We say $(X,J)$ satisfies the $\partial \overline\partial $-Lemma if
$$Im(\partial)\cap Ker(\overline\partial)=Im(\overline\partial)\cap Ker(\partial)=Im(\partial \overline\partial).$$
\end{definition}

 Now we give some weaker definitions:
\begin{definition}
If for any $\partial \varphi \in U^k(X)$ with $\overline\partial \partial \varphi=0$, the equation ~$$\overline\partial \sigma=\partial \varphi$$
 has a solution $\sigma \in U^{k-1}(X)$ ~(a $\partial$ -exact solution $\sigma :=\partial \sigma_1\in U^{k-1}(X)$), we denote as $X \in \mathbb{S}^k$
 ~($X \in \mathbb{B}^k$);
 similarly, if for any $ \varphi \in U^{k+1}(X)$ with $\overline\partial  \varphi=0$, the equation ~$$\overline\partial \sigma=\partial \varphi$$
 has a solution $\sigma \in U^{k-1}(X)$ ~(a $\partial$ -exact solution $\sigma :=\partial \sigma_1\in U^{k-1}(X)$), we denote as $X \in \mathcal{S}^k$
 ~($X \in \mathcal{B}^k$).

\end{definition}
Thus, we have the following relationship between the two definitions above:
\begin{lemma}\label{dd}
If $X$ satisfies $\partial \overline\partial $-lemma, we have that $X\in \mathbb{B}^k$ for any $-n\le k\le n $.

\end{lemma}

Now , we give the following lemmas which will be used in the  paper later.
\begin{proposition}\label{eqn02}

Let $(X,G)$ be a compact generalized Hermitian  manifold. If the $\partial\overline\partial$-equation

\begin{eqnarray}\label{eqn01}
\partial\overline\partial x=y
\end{eqnarray}

has a solution, then
$x=(\partial\overline\partial)^{\ast}G_{BC}y$ is also a solution which has the minimum $L^2$-norm with respect to the Born-Infeld inner product definite by Formula \ref{BIinnerproduct}.

\end{proposition}

\begin{proof}
Let $x$ be a solution of Equation \ref{eqn01}. We can decomposition $x$ by
\begin{eqnarray*}
x=x_1+x_2+x_3,
\end{eqnarray*}

where $x_1 \in Ker \Delta_{A},  x_2 \in Im \partial + Im \overline\partial, x_3 \in Im (\partial\overline\partial)^{\ast}$.

Then
\begin{eqnarray*}
\partial\overline\partial x_1&=&0,
\partial\overline\partial x_2=0,
\partial\overline\partial x=\partial\overline\partial x_3=y.
\end{eqnarray*}

Also, since $x_3 \in Im (\partial\overline\partial)^{\ast}$, we have that
\begin{eqnarray*}
\partial^{\ast} x_3&=&\overline\partial^{\ast}x_3=0,.
\end{eqnarray*}

Then

\begin{eqnarray*}
(\partial\overline\partial)^{\ast}y&=& (\partial\overline\partial)^{\ast}(\partial\overline\partial)x_3=\Delta_{A} x_3,\\
G_{A}(\partial\overline\partial)^{\ast}y&=&G_{A}\Delta_{A} x_3 =x_3. ~~(x_3 \in Im (\partial\overline\partial)^{\ast})
\end{eqnarray*}

Thus

\begin{eqnarray*}
x_3&=&G_{A}(\partial\overline\partial)^{\ast}y\\
&=&(\partial\overline\partial)^{\ast} G_{BC}y.
\end{eqnarray*}

Also, for any solution $x,$ we have that
\begin{eqnarray*}
||x||^2&=& ||x_1||^2+||x_2||^2+||x_3||^2 \ge ||x_3||^2=||(\partial\overline\partial)^{\ast} G_{BC}y||^2.
\end{eqnarray*}

\end{proof}

\begin{proposition}\label{d-closed representation}

Let $(X,G)$ be a compact generalized Hermitian   manifold and assume that $X\in \mathbb{B}^{k-1}$. For any $[\sigma] \in H_{\overline\partial }^k (X)$. Choose a representation of $[\sigma]$ and
we still denote it as $\sigma \in H_{\overline\partial }^k (X)$.
Then  there exists some $\beta_{\sigma} \in U^{k-1}(X)$,  such that $\gamma _{\sigma} := \sigma +\overline\partial \beta_{\sigma}$ is also a representation of $[\sigma]$ with $d \gamma _{\sigma}=0$.

\end{proposition}

\begin{proof}
If $\gamma _{\sigma}$ exists, we can get that there exists some $\beta_{\sigma} \in U^{k-1}(X) $ such that ~$$\gamma _{\sigma}=\sigma + \overline\partial \beta_{\sigma}.$$
Thus,
\begin{eqnarray*}
0&=&\partial \gamma _{\sigma}~(d \gamma _{\sigma}=0)\\
&=&\partial \sigma + \partial\overline\partial \beta_{\sigma}.
\end{eqnarray*}
That is
\begin{eqnarray*}
\partial \sigma =- \partial\overline\partial \beta_{\sigma}.
\end{eqnarray*}
So by the assumption that $X\in \mathbb{B}^{k-1}$ and Proposition \ref{eqn02} above, we can get a solution $\beta_{\sigma}=-(\partial \overline\partial )^{\ast} G_{BC} \partial \sigma \in U^{k-1}(X)$
and $\gamma _{\sigma} = \sigma +\overline\partial \beta_{\sigma}$ is desired.
\end{proof}

\section{Criterion for holomorphism}

Let $\pi: \mathcal{X} \to \Delta \subset \mathbb{C}^1 $ be a smooth  family of generalized complex manifolds  where the centre manifold $\pi^{-1}(0):=X_0:=X$ is a compact  generalized Hermitian manifold.
In this section , we get a criterion formula for whether a differential form is holomorphic with respect to the new generalized complex structure $(X_{\epsilon}, J_{\epsilon})$ induced by $\epsilon$. The method we use is parallel to
that in \cite{RaoWanZhao03}.
We first define a map between $\wedge ^k L^{\ast} \cdot U^{-n}(X) $ and   $ \wedge ^k L^{\ast}_{\epsilon} \cdot U^{-n}(X_{\epsilon})$.
\begin{definition}
\begin{eqnarray*}
\mathcal{E}: \wedge ^k L^{\ast} \cdot U^{-n}(X)    &\to&  \wedge ^k L^{\ast}_{\epsilon} \cdot U^{-n}(X_{\epsilon}) ,\\
\sigma  &\mapsto & \mathcal{E}(\sigma)
\end{eqnarray*}
where  \begin{eqnarray*}
\sigma &:=&\frac{1}{k!}\sigma_{i_1\cdots i_k}l^{i_1}\cdot \cdots \cdot l^{i_k}\cdot \rho_0,~~l^{i_{\alpha}}\in L^{\ast},\rho_0\in U^{-n}(X),\\
 \mathcal{E}(\sigma)&:=& \frac{1}{k!}\sigma_{i_1\cdots i_k} (1+\epsilon^{\ast})(l^{i_1})\cdot \cdots \cdot (1+\epsilon^{\ast})(l^{i_k})\cdot (e^{\epsilon}\cdot \rho_0),\\
   e^{\epsilon}\cdot \rho_0&:=& \sum_{i\ge 0} \frac{1}{i!}\epsilon^i \cdot \rho_0, ~\epsilon^i:=\overbrace{\epsilon\cdot \epsilon \cdot \cdots \cdot \epsilon \cdot}^i.\\
\end{eqnarray*}
\end{definition}

About $e^{\epsilon}$ and $\mathcal{E}$, we have the following two propositions.
\begin{proposition}\label{iso01}
\begin{eqnarray*}
e^{\epsilon}\cdot : ( L\oplus L^{\ast} )\cdot \wedge ^{\bullet}T^{\ast}_X  &\to& \wedge ^{\bullet}T^{\ast}_X,\\
l\cdot \rho &\mapsto &e^{\epsilon}\cdot (l\cdot \rho)= (1+\epsilon)(l) \cdot (e^{\epsilon}\cdot \rho)\\
e^{\epsilon^{\ast}}\cdot : ( L\oplus L^{\ast} )\cdot \wedge ^{\bullet}T^{\ast}_X  &\to& \wedge ^{\bullet}T^{\ast}_X.\\
l\cdot \rho &\mapsto &e^{\epsilon^{\ast}}\cdot (l\cdot \rho)= (1+\epsilon^{\ast})(l) \cdot (e^{\epsilon^{\ast}}\cdot \rho)
\end{eqnarray*}

\end{proposition}

\begin{proof}
Let $\epsilon:=a\cdot b \in \wedge^2 L^{\ast}.$
By Formula \ref{wedge}, we have that
\begin{eqnarray*}
\epsilon \cdot l \cdot \rho &:=& a\cdot b\cdot l\cdot \rho\\
&=&a \cdot (b(l) \rho  -l\cdot b\cdot \rho)\\
&=&b(l) a \cdot \rho  -a(l)  b \cdot \rho + l\cdot a \cdot b\cdot \rho\\
&:=&\epsilon(l) \cdot \rho + l\cdot \epsilon \cdot \rho.
\end{eqnarray*}

Also, we have that $\epsilon(\epsilon((L\oplus L^{\ast}))= \epsilon(\epsilon (L))\subset \epsilon (L^{\ast})=0.$
So we have that
\begin{eqnarray*}
\frac{1}{2!}\epsilon ^2 \cdot l\cdot \rho &=&\frac{1}{2!}\epsilon \cdot (\epsilon(l)\cdot \rho + l\cdot \epsilon \cdot \rho)\\
&=&\frac{1}{2!}(\epsilon(\epsilon(l))\cdot \rho+\epsilon(l)\cdot\epsilon \cdot \rho+\epsilon(l) \cdot \epsilon \cdot \rho+ l\cdot \epsilon^2 \cdot \rho)\\
&=&\epsilon(l)\cdot\epsilon \cdot \rho+\frac{1}{2!}l\cdot \epsilon^2 \cdot \rho
\end{eqnarray*}

By induction on $k$, if

\begin{eqnarray*}
\frac{1}{k!}\epsilon ^k \cdot l\cdot \rho=\frac{1}{(k-1)!}\epsilon(l)\cdot\epsilon^{k-1} \cdot \rho+\frac{1}{k!}l\cdot \epsilon^k \cdot \rho,
\end{eqnarray*}

then
\begin{eqnarray*}
\frac{1}{(k+1)!}\epsilon ^{k+1} \cdot l\cdot \rho&=&\frac{1}{k+1}\epsilon \cdot (
\frac{1}{(k-1)!}\epsilon(l)\cdot\epsilon^{k-1} \cdot \rho+\frac{1}{k!}l\cdot \epsilon^k \cdot \rho),\\
&=&\frac{1}{(k-1)!(k+1)}\epsilon(\epsilon(l))\cdot\epsilon^{k-1} \cdot \rho +\frac{1}{(k-1)!(k+1)}\epsilon(l) \cdot \epsilon^k \cdot
\rho\\&&
 + \frac{1}{(k+1)!}l\cdot \epsilon^{k+1}\cdot \rho+ \frac{1}{(k+1)!}\epsilon(l)\cdot \epsilon^k \cdot \rho\\
&=&\frac{1}{(k+1)!}l\cdot \epsilon^{k+1}\cdot \rho+ \frac{1}{k!}\epsilon(l)\cdot \epsilon^k \cdot \rho.
\end{eqnarray*}

Thus \begin{eqnarray*}
e^{\epsilon}\cdot l\cdot \rho &=&(1+\epsilon)(l) \cdot (e^{\epsilon}\cdot \rho).
\end{eqnarray*}

The proof of the second formula is similar to the first one.

\end{proof}

Thus, by Proposition \ref{iso01} above and $\epsilon(L^{\ast})=\epsilon^{\ast}(L)=0, $  we have the following equivalent definition of $\mathcal{E}$:
\begin{eqnarray*}
\mathcal{E}: \wedge ^k (L \oplus L^{\ast}) \cdot  \wedge^{\bullet}T_X^{\ast}    &\to&  \wedge ^k( L_{\epsilon} \oplus L^{\ast}_{\epsilon}) \cdot \wedge^{\bullet}T_{X_{\epsilon}}^{\ast} ,\\
\sigma  &\mapsto & \mathcal{E}(\sigma)
\end{eqnarray*}
where  \begin{eqnarray*}
\sigma &:=&\frac{1}{k!}\sigma_{i_1\cdots i_k}(l^{i_1}+l_{i_1})\cdot \cdots \cdot (l^{i_k}+l_{i_k})\cdot \rho_1,~~l_{i_{\alpha}}\in L , ~l^{i_{\alpha}}\in L^{\ast},\rho_1\in \wedge^{\bullet}T_X^{\ast},\\
 \mathcal{E}(\sigma)&:=& \frac{1}{k!}\sigma_{i_1\cdots i_k} (1+\epsilon+\epsilon^{\ast})(l^{i_1}+l_{i_1})\cdot \cdots \cdot (1+\epsilon+\epsilon^{\ast})(l^{i_k}+l_{i_k})\cdot e^{\epsilon+\epsilon^{\ast}}\cdot  \rho_1.
\end{eqnarray*}

\begin{proposition}\label{isoeE}
$e^{\epsilon}\cdot, \mathcal{E}$ are isomorphism.
\end{proposition}

\begin{proof}

\begin{eqnarray*}
e^{-\epsilon}\cdot e^{\epsilon}\cdot &=&(\sum_{i=0}^{\infty} \frac{1}{i!}(-\epsilon)^i)\cdot(\sum_{j=0}^{\infty} \frac{1}{j!}(\epsilon)^j)\cdot\\
&=&\sum_{p=0}^{\infty}\sum_{i+j=p}\frac{1}{i!}\frac{1}{j!}(-\epsilon)^i\cdot(\epsilon)^j\cdot\\
&=&1+\sum_{p=1}^{\infty} \sum_{i=0}^p \frac{1}{i!}\frac{1}{(p-i)!}(-1)^i(\epsilon)^{p}\cdot\\
&=&1.
\end{eqnarray*}

The last equivalence since

\begin{eqnarray*}
&&\frac{1}{p!}-\frac{1}{(p-1)!}\frac{1}{1!}+\frac{1}{(p-2)!}\frac{1}{2!}-\cdots +(-1)^p \frac{1}{p!}\\
&=&\frac{1}{p!}(C_p^p-C_p^{p-1}+C_p^{p-2}-\cdots +(-1)^p C_p^0)=0.
\end{eqnarray*}

Assuming that $ \{l^1, \cdots, l^{2n}\}$ is a basis of $L^{\ast}$ and $\{\rho_0\}$ is a basis of $U^{-n}(X)$,
we have that  $\{(1+\epsilon^{\ast})l^1, \cdots, (1+\epsilon^{\ast})l^{2n}\}$ is a basis of $L_{\epsilon}^{\ast}$ and $\{e^{\epsilon}\rho_0\}$ is a basis of $U^{-n}(X_{\epsilon})$. Then  for any $\sigma \in U^{k}(X_{\epsilon})$, $\sigma$ can be locally represented as
$\sigma = \frac{1}{k!} \sigma_{i_1\cdots i_k} (1+\epsilon^{\ast})l^{i_1}\dot \cdots \cdot (1+\epsilon^{\ast})l^{i_k}\cdot e^{\epsilon}\rho_0$.

We define $\mathcal{E}^{-1}$ by
\begin{eqnarray*}
\mathcal{E}^{-1}:  \wedge ^k L^{\ast}_{\epsilon} \cdot U^{-n}(X_{\epsilon})  &\to& \wedge ^k L^{\ast} \cdot U^{-n}(X)   ,\\
\sigma  &\mapsto & \mathcal{E}^{-1}(\sigma)
\end{eqnarray*}
where  \begin{eqnarray*}
\sigma &:=& \frac{1}{k!} \sigma_{i_1\cdots i_k} (1+\epsilon^{\ast})l^{i_1}\dot \cdots \cdot (1+\epsilon^{\ast})l^{i_k}\cdot e^{\epsilon}\rho_0,\\
 \mathcal{E}^{-1}(\sigma)&:=&  \frac{1}{k!} \sigma_{i_1\cdots i_k} (1-\epsilon^{\ast})(1+\epsilon^{\ast})l^{i_1}\dot \cdots \cdot (1-\epsilon^{\ast}) (1+\epsilon^{\ast})l^{i_k}\cdot e^{-\epsilon} \cdot e^{\epsilon}\rho_0.
\end{eqnarray*}
And we have that $\mathcal{E}^{-1} \circ \mathcal{E}=1.$
Similarly, we have that $e^{\epsilon}\cdot e^{-\epsilon}\cdot =\mathcal{E} \circ \mathcal{E}^{-1} \circ=1$.

\end{proof}

Now we introduce the  following three propositions which will be used in computation of the  criterion formula.
\begin{proposition}\label{expansion01}
\begin{eqnarray*}
e^{-\epsilon}\cdot \mathcal{E} (\sigma) =(1+\epsilon^{\ast} - \epsilon\epsilon^{\ast}) ( \sigma),
\end{eqnarray*}
where for any $\sigma:=\frac{1}{k!}\sigma_{i_1\cdots i_k}l^{i_1}\cdot \cdots \cdot l^{i_k}\cdot \rho_0$,
\begin{eqnarray*}
&&(1+\epsilon^{\ast} - \epsilon\epsilon^{\ast})\sigma:=(1+\epsilon^{\ast} - \epsilon\epsilon^{\ast})
(\frac{1}{k!}\sigma_{i_1\cdots i_k}l^{i_1}\cdot \cdots \cdot l^{i_k}\cdot \rho_0)\\
&:=&\frac{1}{k!}\sigma_{i_1\cdots i_k}(1+\epsilon^{\ast} - \epsilon\epsilon^{\ast})(l^{i_1})\cdot \cdots \cdot (1+\epsilon^{\ast} - \epsilon\epsilon^{\ast})(l^{i_k})\cdot \rho_0\\
&=&\frac{1}{k!}\sigma_{i_1\cdots i_k}(1+\epsilon^{\ast} - \epsilon\epsilon^{\ast})(l^{i_1})\cdot \cdots \cdot (1+\epsilon^{\ast} - \epsilon\epsilon^{\ast})(l^{i_k})\cdot e^{\epsilon^{\ast}- \epsilon \epsilon^{\ast}}\cdot \rho_0.\\
\end{eqnarray*}

Here we use the fact  that $\epsilon^{\ast} \cdot U^{-n}(X)=0.$

\end{proposition}
\begin{proof}
The proof follows from straightforward computation. For any $\sigma:=\frac{1}{k!}\sigma_{i_1\cdots i_k}l^{i_1}\cdot \cdots \cdot l^{i_k}\cdot \rho_0$, by Proposition \ref{iso01} above, we have that
\begin{eqnarray*}
e^{-\epsilon}\cdot \mathcal{E} (\sigma) &=&e^{-\epsilon}\cdot \mathcal{E}  (\frac{1}{k!}\sigma_{i_1\cdots i_k}l^{i_1}\cdot \cdots \cdot l^{i_k}\cdot \rho_0)\\
&=&\frac{1}{k!}\sigma_{i_1\cdots i_k}e^{-\epsilon}\cdot (1+\epsilon^{\ast})(l^{i_1})\cdot \cdots \cdot (1+\epsilon^{\ast})(l^{i_k})\cdot (e^{\epsilon}\cdot \rho_0)\\
&=&\frac{1}{k!}\sigma_{i_1\cdots i_k}(1-\epsilon)  ((1+\epsilon^{\ast})(l^{i_1}))\cdot \cdots \cdot(1-\epsilon) ((1+\epsilon^{\ast})(l^{i_k}))\cdot e^{-\epsilon}\cdot(e^{\epsilon}\cdot \rho_0)\\
&=&\frac{1}{k!}\sigma_{i_1\cdots i_k}(1-\epsilon+\epsilon^{\ast}-\epsilon\epsilon^{\ast})(l^{i_1})\cdot \cdots \cdot(1-\epsilon+\epsilon^{\ast}-\epsilon\epsilon^{\ast})(l^{i_k})\cdot \rho_0\\
&=&\frac{1}{k!}\sigma_{i_1\cdots i_k}(1+\epsilon^{\ast}-\epsilon\epsilon^{\ast})(l^{i_1})\cdot \cdots \cdot(1+\epsilon^{\ast}-\epsilon\epsilon^{\ast})(l^{i_k})\cdot \rho_0.
\end{eqnarray*}

The last equality holds since $\epsilon( L^{\ast})=0.$
\end{proof}

\begin{proposition}\label{expansion02} Assume that $||\epsilon||_{L^{\infty}}<1$, we have that
\begin{eqnarray*}
\mathcal{E}^{-1}\circ e^{\epsilon}(\sigma)=(-\epsilon^{\ast}(1-\epsilon\epsilon^{\ast})^{-1}+(1-\epsilon\epsilon^{\ast})^{-1})( \sigma).
\end{eqnarray*}
\end{proposition}
\begin{proof}

Since

\begin{eqnarray*}
\epsilon \epsilon^{\ast} ( l^k)&:=&\epsilon (\epsilon^{qk}l_q)\\
&=&\epsilon^{qk}\epsilon_{pq}\epsilon^p,\\
\end{eqnarray*}

we have that
\begin{eqnarray*}
\epsilon \epsilon^{\ast} \left(\begin{array}{c}
l^1\\
\cdots \\
l^{2n}
\end{array}\right)&:=&[\epsilon^{\ast}][ \epsilon]\left(\begin{array}{c}
l^1\\
\cdots \\
l^{2n}
\end{array}\right).
\end{eqnarray*}

Similarly,
\begin{eqnarray*}
 \epsilon^{\ast}\epsilon \epsilon^{\ast} \left(\begin{array}{c}
l^1\\
\cdots \\
l^{2n}
\end{array}\right)&:=&[\epsilon^{\ast}][ \epsilon] [\epsilon^{\ast}]\left(\begin{array}{c}
l_1\\
\cdots \\
l_{2n}
\end{array}\right).
\end{eqnarray*}

Since
\begin{eqnarray*}
l^p=\xi^i(l^p)\xi_i+\xi_i(l^p)\xi^i,
\end{eqnarray*}

by the assumption that $||\epsilon||_{L^{\infty}}<1$, Proposition \ref{iso01}, Proposition \ref{unholpart} and Proposition \ref{001}, we have that

\begin{eqnarray*}
\left(\begin{array}{c}
l^1\\
\cdots \\
l^{2n}
\end{array}\right)&=&
\left(\begin{array}{ccc}
\xi^1(l^1) &  \cdots &  \xi^{2n}(l^1)\\
\cdots &  \cdots  &   \cdots\\
\xi^1(l^{2n}) &  \cdots &  \xi^{2n}(l^{2n})
\end{array}\right)
\left(\begin{array}{c}
\xi_1\\
\cdots \\
\xi_{2n}
\end{array}\right)+
\left(\begin{array}{ccc}
\xi_1(l^1) &  \cdots &  \xi_{2n}(l^1)\\
\cdots &  \cdots  &   \cdots\\
\xi_1(l^{2n}) &  \cdots &  \xi_{2n}(l^{2n})
\end{array}\right)
\left(\begin{array}{c}
\xi^1\\
\cdots \\
\xi^{2n}
\end{array}\right)\\
&=&
\left(\begin{array}{ccc}
\xi^1(l^1) &  \cdots &  \xi^{2n}(l^1)\\
\cdots &  \cdots  &   \cdots\\
\xi^1(l^{2n}) &  \cdots &  \xi^{2n}(l^{2n})
\end{array}\right)
\left(\begin{array}{ccc}
l^1(\xi_1) &  \cdots &  l^{2n}(\xi_1)\\
\cdots &  \cdots  &   \cdots\\
l^1(\xi_{2n}) &  \cdots &  l^{2n}(\xi_{2n})
\end{array}\right)(1+\epsilon)
\left(\begin{array}{c}
l_1\\
\cdots \\
l_{2n}
\end{array}\right)\\
&&+
\left(\begin{array}{ccc}
\xi_1(l^1) &  \cdots &  \xi_{2n}(l^1)\\
\cdots &  \cdots  &   \cdots\\
\xi_1(j^{2n}) &  \cdots &  \xi_{2n}(l^{2n})
\end{array}\right)
\left(\begin{array}{ccc}
l_1(\xi^1) &  \cdots &  l_{2n}(\xi^1)\\
\cdots &  \cdots  &   \cdots\\
l_1(\xi^{2n}) &  \cdots &  l_{2n}(\xi^{2n})
\end{array}\right)(1+\epsilon^{\ast})
\left(\begin{array}{c}
l^1\\
\cdots \\
l^{2n}
\end{array}\right)\\
&=&
L^{\ast}_{\epsilon}(L^{\ast})L^{\ast}(L_{\epsilon})
(1+\epsilon)
\left(\begin{array}{c}
l_1\\
\cdots \\
l_{2n}
\end{array}\right)+
L_{\epsilon}(L^{\ast})L(L^{\ast}_{\epsilon})(1+\epsilon^{\ast})
\left(\begin{array}{c}
l^1\\
\cdots \\
l^{2n}
\end{array}\right)\\
&=&(-[\epsilon^{\ast}](1-[\epsilon][\epsilon^{\ast}])^{-1})(1+\epsilon)
\left(\begin{array}{c}
l_1\\
\cdots \\
l_{2n}
\end{array}\right)+(1-[\epsilon^{\ast}][\epsilon])^{-1}(1+\epsilon^{\ast})
\left(\begin{array}{c}
l^1\\
\cdots \\
l^{2n}
\end{array}\right)\\
&=&(1+\epsilon)(-[\epsilon^{\ast}](1-[\epsilon][\epsilon^{\ast}])^{-1})
\left(\begin{array}{c}
l_1\\
\cdots \\
l_{2n}
\end{array}\right)+(1+\epsilon^{\ast})(1-[\epsilon^{\ast}][\epsilon])^{-1}
\left(\begin{array}{c}
l^1\\
\cdots \\
l^{2n}
\end{array}\right)\\
&=&(1+\epsilon)(-[\epsilon^{\ast}]\sum_{i\ge 0} ([\epsilon][\epsilon^{\ast}])^i )
\left(\begin{array}{c}
l_1\\
\cdots \\
l_{2n}
\end{array}\right)+(1+\epsilon^{\ast})(\sum_{i\ge 0} ([\epsilon^{\ast}][\epsilon])^i)
\left(\begin{array}{c}
l^1\\
\cdots \\
l^{2n}
\end{array}\right)\\
&=&(1+\epsilon)(-\sum_{i\ge 0} (\epsilon^{\ast}\epsilon)^i \epsilon^{\ast} )
\left(\begin{array}{c}
l^1\\
\cdots \\
l^{2n}
\end{array}\right)+(1+\epsilon^{\ast})(\sum_{i\ge 0} (\epsilon \epsilon^{\ast} )^i)
\left(\begin{array}{c}
l^1\\
\cdots \\
l^{2n}
\end{array}\right)\\
&=&(1+\epsilon)(- (1-\epsilon^{\ast} \epsilon)^{-1} \epsilon^{\ast})
\left(\begin{array}{c}
l^1\\
\cdots \\
l^{2n}
\end{array}\right)+(1+\epsilon^{\ast}) (1-\epsilon \epsilon^{\ast} )^{-1}
\left(\begin{array}{c}
l^1\\
\cdots \\
l^{2n}
\end{array}\right)\\
&=&(1+\epsilon+\epsilon^{\ast})(- (1-\epsilon^{\ast} \epsilon)^{-1}\epsilon^{\ast} + (1-\epsilon \epsilon^{\ast} )^{-1})
\left(\begin{array}{c}
l^1\\
\cdots \\
l^{2n}
\end{array}\right)\\
&=&(1+\epsilon+\epsilon^{\ast})(- \epsilon^{\ast}(1-\epsilon \epsilon^{\ast})^{-1} + (1-\epsilon \epsilon^{\ast} )^{-1})
\left(\begin{array}{c}
l^1\\
\cdots \\
l^{2n}
\end{array}\right).
\end{eqnarray*}

Thus, for any $\sigma:=\frac{1}{k!}\sigma_{i_1\cdots i_k}l^{i_1}\cdot \cdots \cdot l^{i_k}\cdot \rho_0$, we have that

\begin{eqnarray*}
\mathcal{E}^{-1} \circ e^{\epsilon} \cdot \sigma &:=& \mathcal{E}^{-1} \circ e^{\epsilon} \cdot (\frac{1}{k!}\sigma_{i_1\cdots i_k}l^{i_1}\cdot \cdots \cdot l^{i_k}\cdot \rho_0)\\
&=&\mathcal{E}^{-1} \circ (\frac{1}{k!}\sigma_{i_1\cdots i_k}  l^{i_1}\cdot \cdots \cdot l^{i_k}\cdot (e^{\epsilon}\cdot \rho_0))~~(\mbox{since} ~\epsilon( L^{\ast})=0)\\
&=& \frac{1}{k!}\sigma_{i_1\cdots i_k} \mathcal{E}^{-1}( l^{i_1}\cdot \cdots \cdot l^{i_k}\cdot  e^{\epsilon}\cdot\rho_0)\\
&=& \frac{1}{k!}\sigma_{i_1\cdots i_k} (- \epsilon^{\ast}(1-\epsilon \epsilon^{\ast})^{-1}+(1-\epsilon \epsilon^{\ast} )^{-1})( l^{i_1})\cdot \\
&&\cdots \cdot (- \epsilon^{\ast}(1-\epsilon \epsilon^{\ast})^{-1}+(1-\epsilon \epsilon^{\ast} )^{-1})(l^{i_k}) \cdot\rho_0\\
\end{eqnarray*}

\end{proof}

\begin{theorem}\label{expansion03} Assume that $||\epsilon||_{L^{\infty}}<1$, we have that
\begin{eqnarray*}
d\circ \mathcal{E}=\mathcal{E}\circ (-\epsilon^{\ast}(1-\epsilon\epsilon^{\ast})^{-1}+(1-\epsilon\epsilon^{\ast})^{-1})
((d+[\partial , \epsilon\cdot])\circ  (1+\epsilon^{\ast} - \epsilon\epsilon^{\ast}) ( \sigma)).
\end{eqnarray*}
\end{theorem}
\begin{proof}
By Proposition \ref{expansion01} and Proposition \ref{expansion02}, we have that
\begin{eqnarray*}
d\circ \mathcal{E}(\sigma)&=&d\circ e^{\epsilon}\cdot e^{-\epsilon}\cdot\mathcal{E}(\sigma)\\
&=&e^{\epsilon}\cdot (d+[\partial , \epsilon\cdot])\circ  e^{-\epsilon}\cdot\mathcal{E}(\sigma) ~ ~(\mbox{since ~Lemma}~ \ref{holcri})\\
&=&\mathcal{E}\circ \mathcal{E}^{-1} \circ e^{\epsilon}\cdot (d+[\partial , \epsilon\cdot])\circ  e^{-\epsilon}\cdot\mathcal{E}(\sigma)\\
&=&\mathcal{E}\circ (-\epsilon^{\ast}(1-\epsilon\epsilon^{\ast})^{-1}+(1-\epsilon\epsilon^{\ast})^{-1})
(d+[\partial , \epsilon\cdot])\circ  (1+\epsilon^{\ast} - \epsilon\epsilon^{\ast}) ( \sigma).
\end{eqnarray*}

\end{proof}

Now, we get the criterion formula for whether a differential form is holomorphic with respect to the generalized complex structure $J_{\epsilon}$ induced by $\epsilon$.
\begin{theorem}\label{k-cri}
Assume that $||\epsilon||_{L^{\infty}}<1$, we have that
\begin{eqnarray*}
\overline\partial_t(\mathcal{E} (\sigma))=0
\Leftrightarrow ([\partial, \epsilon\cdot]+\overline\partial)
\circ (1 - \epsilon\epsilon^{\ast}) ( \sigma)=0,
\end{eqnarray*}

where $\overline\partial_t$ is the $\overline\partial$-operator on $X_{\epsilon}:=X_{\epsilon(t)}$.
\end{theorem}

\begin{proof}
By definition, we have that

\begin{eqnarray*}
\partial &:&  U^k(X) \to U^{k-1}(X),\\
\overline\partial +[\partial, \epsilon\cdot]&:&U^k(X) \to U^{k+1}(X),\\
(1+\epsilon^{\ast}-\epsilon\epsilon^{\ast})&:&  U^k(X) \to \oplus_{p\ge 0} U^{k-2p}(X),\\
-\epsilon^{\ast}(1-\epsilon\epsilon^{\ast})^{-1}+(1-\epsilon\epsilon^{\ast})^{-1}&:& U^k(X) \to \oplus_{p\ge 0} U^{k-2p}(X).\\
\end{eqnarray*}
Here for any $\sigma:=\frac{1}{k!}\sigma_{i_1\cdots i_k}l^{i_1}\cdot \cdots \cdot l^{i_k}\cdot \rho_0$,
$$\epsilon^{\ast}(\sigma):=\frac{1}{k!}\sigma_{i_1\cdots i_k}\epsilon^{\ast}(l^{i_1})\cdot \cdots \cdot \epsilon^{\ast}(l^{i_k})\cdot e^{\epsilon^{\ast}}\cdot \rho_0.$$
And similar to $(1-\epsilon\epsilon^{\ast})^{-1}, ~(1-\epsilon\epsilon^{\ast}), ~-\epsilon^{\ast}(1-\epsilon\epsilon^{\ast})^{-1}.$

Then, by Proposition \ref{isoeE}, Proposition \ref{k-cri} above and comparing the type of the differential form,  for any $\sigma \in U^k(X),$ we have that
\begin{eqnarray*}
&&\overline\partial_t \circ \mathcal{E} (\sigma)\\
&:=&d \circ \mathcal{E} (\sigma)\mid _{U^{k+1}(X_{\epsilon(t)})} \\
&=& \mathcal{E} \circ ( (-\epsilon^{\ast}(1-\epsilon\epsilon^{\ast})^{-1}+(1-\epsilon\epsilon^{\ast})^{-1})
((d+[\partial , \epsilon\cdot])\circ  (1+\epsilon^{\ast} - \epsilon\epsilon^{\ast}) ( \sigma)) \mid _{ U^{k+1}(X)})\\
&=&\mathcal{E} \circ  ((1-\epsilon\epsilon^{\ast})^{-1})
(\overline\partial+[\partial , \epsilon\cdot])\circ  (1 - \epsilon\epsilon^{\ast}) ( \sigma).
\end{eqnarray*}

Since the assumption that   $||\epsilon||_{L^{\infty}}<1$, we have that $\mathcal{E} \circ  ((1-\epsilon\epsilon^{\ast})^{-1})$ is invertible and thus

\begin{eqnarray*}
&&\overline\partial_t \circ \mathcal{E} (\sigma)=0\\
&\Leftrightarrow&
(\overline\partial+[\partial , \epsilon\cdot])\circ  (1 - \epsilon\epsilon^{\ast}) ( \sigma)=0.
\end{eqnarray*}

\end{proof}

\section{Locally extensions}
In this section  , we get the following local extensions of $\overline\partial$-closed forms on a smooth family $\pi: \mathcal{X} \to \Delta \subset \mathbb{C}^1 $ of compact generalized Hermitian manifolds. The method we used is parallel to that in \cite{RaoWanZhao03}  which originally came from \cite{To, Ti, Cl, Ra, Sun01, Sun02, Zh, RaoZhao04, RaoWanZhao02, RaoWanZhao01}. This is also an extension of  local extensions of canonical forms in \cite{KW02}.

\begin{theorem}\label{construction}
Let $(X,G)$ be a compact Hermitian generalized complex manifold. Assume that $X\in \mathbb{B}^{k-1} \cap \mathbb{S}^{k+1}.$ Then for any $\sigma_{00} \in H_{\overline\partial }^k (X) ~(-n\le k \le n)$,  we can choose

 \begin{eqnarray*}
\sigma_t = \sigma _{00}+\sum_{i,j\ge 1} t^i \bar{t}^j \sigma_{ij} \in U^k(X),
\end{eqnarray*}
such that $\mathcal{E}(\sigma_t) \in U^k(X_{\epsilon(t)})$ and  $\overline\partial_t \circ \mathcal{E}(\sigma_t)=0$, where $\overline\partial_t$ is
the $\overline\partial$-operator on $X_{\epsilon(t)}$.

\end{theorem}

\begin{proof}
Step 1. We construct $\sigma_t$.\\

Set $\tilde{\sigma_t}:=(1-\epsilon\epsilon^{\ast})(\sigma_t).$ By Theorem \ref{k-cri},

\begin{eqnarray*}
\overline\partial_t(\mathcal{E} (\sigma_t))=0
&\Leftrightarrow& ([\partial, \epsilon\cdot]+\overline\partial)
\circ (1 - \epsilon\epsilon^{\ast}) ( \sigma_t)=0\\
&\Leftrightarrow& ([\partial, \epsilon\cdot]+\overline\partial)
\circ \tilde{\sigma_t}=0.
\end{eqnarray*}

We resolve the equation as
\begin{eqnarray*}
\partial (\tilde{\sigma}_t)&=&0,\\
\overline\partial (\tilde{\sigma}_t)&=&-\partial (\epsilon(t) \cdot \tilde{\sigma}_t).
\end{eqnarray*}

Substitute
\begin{eqnarray*}
\epsilon(t):=\sum_{i+j\ge 1} t^i \bar{t}^j\epsilon_{ij},\\
\sigma_t:=\sigma_{00}+\sum_{p+q\ge 1} t^p \bar{t}^q\sigma_{pq},\\
\tilde{\sigma_t}:=\tilde{\sigma}_{00}+\sum_{p+q\ge 1} t^p \bar{t}^q\tilde{\sigma}_{pq}
\end{eqnarray*}
into the above formula and compare the coefficients of $t^p \bar{t}^q$, we get that
\begin{eqnarray}
\partial \tilde{\sigma}_{pq}&=&0,  \label{hol01}\\
\overline\partial \tilde{\sigma}_{pq} &=&-\sum_{i+j=p, k+l=q, i+k\ge 1}\partial (\epsilon_{ik} \cdot \tilde{\sigma}_{jl}).~(p+q\ge 1)\label{hol02}
\end{eqnarray}

Since
\begin{eqnarray*}
\tilde{\sigma_t}&:=&(1-\epsilon\epsilon^{\ast})(\sigma_t)\\
&:=&(1-(\sum_{i+j\ge 1} t^i \bar{t}^j\epsilon_{ij})(\sum_{k+l\ge 1} t^k \bar{t}^l\epsilon^{\ast}_{kl}))(\sigma_{00}+\sum_{p+q\ge 1} t^p \bar{t}^q\sigma_{pq})\\
&=&\sigma_{00} + \sum_{p+q\ge 1} t^p \bar{t}^q\sigma_{pq} -\sum_{i+j\ge 1, k+l\ge 1} t^{i+k} \bar{t}^{j+l}\epsilon_{ij} \epsilon^{\ast}_{kl}(\sigma_{00})\\
&& -\sum_{i+j\ge 1, k+l\ge 1, p+q\ge 1} t^{i+k+p} \bar{t}^{j+l+q}\epsilon_{ij} \epsilon^{\ast}_{kl}(\sigma_{pq} ),
\end{eqnarray*}
we have that $\tilde{\sigma}_{00}=\sigma_{00}$.
By the assumption that $X\in \mathbb{B}^{k-1}$ and Proposition \ref{d-closed representation}, we have that
$
d\tilde{\sigma}_{00}=0,
$
and thus
$
\partial \tilde{\sigma}_{00}=0.
$

Now we construct $\tilde{\sigma}_{pq}~ (p+q\ge 1)$ by induction on $p+q$.
Since
\begin{eqnarray*}
\overline\partial \partial (\epsilon_{10}\cdot \tilde{\sigma}_{00})&=& - \partial \overline\partial(\epsilon_{10}\cdot \tilde{\sigma}_{00})\\
&=&- \partial (d_{L} \epsilon_{10}\cdot \tilde{\sigma}_{00}+\epsilon_{10}\cdot \overline\partial\tilde{\sigma}_{00})\\
&=&0,  ~~(\mbox{since} ~~\epsilon_{10} \in  H_{\overline\partial }^2 (X):=Ker \Delta_{d_L}\cap \wedge^2L^{\ast})\\
\overline\partial \partial (\epsilon_{01}\cdot \tilde{\sigma}_{00})&=&0,
\end{eqnarray*}

by the assumption that $X\in \mathbb{S}^{k+1}, $ we get that the equation
\begin{eqnarray*}
\overline\partial \tilde{\sigma}_{10} &=&-\partial (\epsilon_{10} \cdot \tilde{\sigma}_{00})
\end{eqnarray*}
has a solution $\tilde{\sigma}_{10}^1=-\overline\partial ^{\ast}G_{\overline\partial}\partial (\epsilon_{10} \cdot \tilde{\sigma}_{00})$.

To fulfill the equation $\partial \tilde{\sigma}_{10}=0$, we need to find some $\tilde{\sigma}_{10}^2$, such that
\begin{eqnarray*}
\partial(\tilde{\sigma}_{10}^1 + \overline\partial \tilde{\sigma}_{10}^2)=0.
\end{eqnarray*}
That is,
\begin{eqnarray*}
\partial\tilde{\sigma}_{10}^1 =- \partial\overline\partial \tilde{\sigma}_{10}^2.
\end{eqnarray*}

So by the assumption that $X\in \mathbb{B}^{k-1}  $ and  Proposition \ref{eqn02}, we have a solution $\tilde{\sigma}_{10}^2=-(\partial \overline\partial)^{\ast}G_{BC}\partial \tilde{\sigma}_{10}^1,$ and thus $\tilde{\sigma}_{10}:=\tilde{\sigma}_{10}^1 + \overline\partial \tilde{\sigma}_{10}^2$ satisfies both Formula \ref{hol01} and Formula \ref{hol02}.
In details, $\overline\partial\tilde{\sigma}_{10}:=\overline\partial\tilde{\sigma}_{10}^1 + \overline\partial^2 \tilde{\sigma}_{10}^2
=\overline\partial\tilde{\sigma}_{10}^1=-\partial (\epsilon_{10} \cdot \tilde{\sigma}_{00}).$
Similar, we can get $\tilde{\sigma}_{01}.$

If we have already got $\tilde{\sigma}_{pq}$ which satisfies both Formula \ref{hol01} and \ref{hol02}, where $p+q=1,2,\cdots , N-1$. For $p+q=N,$
we have that
\begin{eqnarray*}
&&-\overline\partial \partial (\sum_{i+j=p, k+l=q, i+k\ge 1}\epsilon_{ik}\cdot  \tilde{\sigma}_{jl})\\
&=&\partial \overline\partial (\sum_{i+j=p, k+l=q, i+k\ge 1}\epsilon_{ik}\cdot  \tilde{\sigma}_{jl})\\
&=&\partial  (\sum_{i+j=p, k+l=q, i+k\ge 1}d_L\epsilon_{ik}\cdot  \tilde{\sigma}_{jl}+ \epsilon_{ik}\cdot  \overline\partial \tilde{\sigma}_{jl})\\
&=&\partial  (\sum_{r+m+j=p, s+n+l=q, r+s\ge 1, m+n\ge 1}\frac{1}{2}[\epsilon_{rs}, \epsilon_{mn}]\cdot  \tilde{\sigma}_{jl}\\
&&+ \sum_{i+j=p, k+l=q, i+k\ge 1}\epsilon_{ik}\cdot  \overline\partial \tilde{\sigma}_{jl})\\
&=&\partial  (\sum_{r+m+j=p, s+n+l=q, r+s\ge 1, m+n\ge 1}\frac{1}{2}[\epsilon_{rs}, \epsilon_{mn}]\cdot  \tilde{\sigma}_{jl}\\
&&- \sum_{i+r+m=p, k+s+n=q, i+k\ge 1, r+s\ge 1}\epsilon_{ik}\cdot  \partial \epsilon_{rs}\tilde{\sigma}_{mn})\\
&=&\partial  (\sum_{r+m+j=p, s+n+l=q, r+s\ge 1, m+n\ge 1}\frac{1}{2}(-\partial (\epsilon_{rs}\cdot \epsilon_{mn} \cdot \tilde{\sigma}_{jl})\\
&&-\epsilon_{rs}\cdot \epsilon_{mn} \cdot \partial\tilde{\sigma}_{jl} + \epsilon_{rs}\cdot \partial(\epsilon_{mn} \cdot \tilde{\sigma}_{jl})+
 \epsilon_{mn}\cdot \partial(\epsilon_{rs} \cdot \tilde{\sigma}_{jl}))\\
&&- \sum_{i+r+m=p, k+s+n=q, i+k\ge 1, r+s\ge 1}\epsilon_{ik}\cdot  \partial \epsilon_{rs}\tilde{\sigma}_{mn})\\
&=&\partial  (\sum_{r+m+j=p, s+n+l=q, r+s\ge 1, m+n\ge 1} \epsilon_{rs}\cdot \partial(\epsilon_{mn} \cdot \tilde{\sigma}_{jl})\\
&&- \sum_{i+r+m=p, k+s+n=q, i+k\ge 1, r+s\ge 1}\epsilon_{ik}\cdot  \partial (\epsilon_{rs}\cdot \tilde{\sigma}_{mn})\\
&=&0.
\end{eqnarray*}

The third equality holds since the integrable condition; the fourth equality holds  since the induction; the fifth equality holds  since Lemma \ref{braket};
and the sixth equality holds  since $\partial ^2=0$ and $\partial \tilde{\sigma}_{pq}=0~(p+q\le N-1)$.

So
\begin{eqnarray*}
\overline\partial \tilde{\sigma}_{pq} =-\sum_{i+j=p, k+l=q, i+k\ge 1}\partial (\epsilon_{ik} \cdot \tilde{\sigma}_{jl}).~(p+q =N)
\end{eqnarray*}

has a solution $\tilde{\sigma}_{pq}^1=-\overline\partial ^{\ast}G_{\overline\partial}\partial (\sum_{i+j=p, k+l=q, i+k\ge 1}\epsilon_{ik} \cdot \tilde{\sigma}_{jl})$. And we can also find $\tilde{\sigma}_{pq}^2=-(\partial \overline\partial)^{\ast}G_{BC}\partial \tilde{\sigma}_{pq}^1,$ such that $\tilde{\sigma}_{pq}:=\tilde{\sigma}_{pq}^1 + \overline\partial \tilde{\sigma}_{pq}^2$ satisfies both Formula \ref{hol01} and Formula \ref{hol02}
by using the same way as that in finding $\tilde{\sigma}_{10}^2$.

Step 2.  We give the regularity of $\tilde{\sigma}_t$ by using the elliptic estimates which is similar to that in Section 8 in Appendix in \cite{Ko1}, Proposition 3.14 in \cite{RaoWanZhao03}, or, Proposition 3.15 in \cite{RaoWanZhao02}.

For the power series $a(|t|):=\sum _{m=1}^{\infty}a_m |t|^m, b(|t|):=\sum _{m=1}^{\infty}b_m |t|^m$ with real positive coefficients, if $a_m \le b_m$ for any $m\in \mathbb{N}$ , we denote it as $a(|t|)\ll b(|t|)$.
Consider an important power series
\begin{eqnarray*}
A(|t|):=\frac{\beta}{16\gamma}\sum_{m=1}^{\infty}\frac{\gamma^{m}}{m^2}|t|^m:=\sum_{m=1}^{\infty}A_m|t|^m,
\end{eqnarray*}
where $\beta, \gamma$ are positive constant.
This power series converges for $|t|< \frac{1}{\gamma}$ and has the following property:
\begin{eqnarray*}
(A(|t|))^2\ll \frac{\beta}{\gamma}A(|t|).
\end{eqnarray*}

First, we prove that $||\tilde{\sigma}_t||_{k+\alpha}\ll A(t)$, where $||\cdot||_{k+\alpha}$ is the H\"older norms.
By $X$ is compact, Lemma 6 in \cite{Li} and Proposition 7.4 in \cite{Ko1}, for any differential form $\sigma$ , we have that
\begin{eqnarray*}
||\partial\sigma||_{k+\alpha}&\le& C_1||\sigma||_{k+1+\alpha},\\
||\overline\partial\sigma||_{k+\alpha}&\le& C_2||\sigma||_{k+1+\alpha},\\
||\partial^{\ast}\sigma||_{k+\alpha}&\le& C_3||\sigma||_{k+1+\alpha},\\
||\overline\partial^{\ast}\sigma||_{k+\alpha}&\le& C_4||\sigma||_{k+1+\alpha},\\
||G_{\overline\partial}\sigma||_{k+\alpha}&\le& C_5||\sigma||_{k-2+\alpha},\\
||G_{BC}\sigma||_{k+\alpha}&\le& C_6||\sigma||_{k-4+\alpha},
\end{eqnarray*}
where $C_i~(1\le i \le 6)$ are positive constants which is only depend on $k,\alpha$ and is independent of the choice of $\sigma$.

We have already known that $||\epsilon(|t|)||_{k+\alpha}\ll  A(|t|)$, where $\epsilon(t)$ is Beltrami differentials(Section 4 in \cite{Go2}, or Theorem 5 in \cite{Li}).
If we choose $\beta:=16(||\tilde{\sigma}_{10}||_{k+\alpha}+||\tilde{\sigma}_{01}||_{k+\alpha})$, we have that $(||\tilde{\sigma}_{10}||_{k+\alpha}+||\tilde{\sigma}_{01}||_{k+\alpha})|t|
\ll A(|t|)$.
By induction on $i=p+q$, \\if $\sum_{i=0}^{N-1}(\sum _{p+q=i}||\tilde{\sigma}_{pq}||_{k+\alpha})|t|^{i}
\ll A(|t|)$, we have that there exists  positive constants $C_{k,\alpha}, K_{k,\alpha}$ which only depend on $k,\alpha$ such that
\begin{eqnarray*}
&&(\sum_{p+q=N}||\tilde{\sigma}_{pq}||_{k+\alpha}) |t|^N\\
&\le&(\sum_{p+q=N}\sum _{i+k=p,j+l=q, i+j\ge1 } ||\overline\partial^{\ast}G_{\overline\partial}\partial
(\epsilon_{ij}\tilde{\sigma}_{kl})||_{k+\alpha}+||\overline\partial \overline\partial^{\ast}\partial^{\ast}G_{BC}\partial (\epsilon_{ij}\tilde{\sigma}_{kl})||_{k+\alpha})|t|^N\\
&\le& C_{k, \alpha}(\sum_{p+q=N}\sum _{i+k=p,j+l=q, i+j\ge 1} ||
\epsilon_{ij}\tilde{\sigma}_{kl}||_{k+\alpha})|t|^N\\
&\le& K_{k,\alpha} \sum_{i+r+j+s=N, i+j\ge1} ||\epsilon_{ij}||_{k+\alpha}||\tilde{\sigma}_{rs}||_{k+\alpha}|t|^N~~(\mbox{Lemma 8.1 on Page 455 in \cite{Ko1}})  \\
&\ll& K_{k,\alpha}(A(|t|))^2\le\frac{K_{k,\alpha}\beta}{\gamma}A(|t|).
\end{eqnarray*}

Hence putting $\gamma=K_{k,\alpha}\beta,$ we obtain that
\begin{eqnarray*}
\sum_{p+q=N}||\tilde{\sigma}_{pq}||_{k+\alpha}|t|^N\ll A(|t|).
\end{eqnarray*}

Then we have that
\begin{eqnarray*}
||\tilde{\sigma}||_{k+\alpha}\ll A(|t|).
\end{eqnarray*}

Now, we prove that $\tilde{\sigma}_t$ is a real analytic family of forms in $t$. By  the proof above, we have that
\begin{eqnarray*}
\tilde{\sigma}_t=-\overline\partial^{\ast}G_{\overline\partial}\partial (\epsilon \tilde{\sigma}_t)+
\overline\partial \overline\partial^{\ast} \partial^{\ast} G_{BC} \partial (\epsilon \tilde{\sigma}_t)+\tilde{\sigma}_{00}.
\end{eqnarray*}

Take $\Delta_{\overline\partial} $ on both sides, we have that
\begin{eqnarray*}
\Delta_{\overline\partial}\tilde{\sigma}_t&=&-\Delta_{\overline\partial}\overline\partial^{\ast}G_{\overline\partial}\partial (\epsilon \tilde{\sigma}_t)+
\Delta_{\overline\partial}\overline\partial \overline\partial^{\ast} \partial^{\ast} G_{BC} \partial (\epsilon \tilde{\sigma}_t)+\Delta_{\overline\partial}\tilde{\sigma}_{00}\\
&=&-\overline\partial^{\ast}\partial (\epsilon \tilde{\sigma}_t)+
\overline\partial \overline\partial^{\ast}  \overline\partial \overline\partial^{\ast} \partial^{\ast} G_{BC} \partial (\epsilon \tilde{\sigma}_t)+\overline\partial \overline\partial^{\ast}\tilde{\sigma}_{00}.
\end{eqnarray*}

Now, for each $l\in \mathbb{N}, $ we choose a smooth function $\eta^l(t)$, such that

\begin{eqnarray*}
\eta^l(t) =
\{
\begin{array}{cc}
1   ,   &  |t| \le (\frac{1}{2}+\frac{1}{2^{l+1}})r;\\
0  ,    &  |t| \ge (\frac{1}{2}+\frac{1}{2^{l}})r.
\end{array}
\end{eqnarray*}

We also choose a partition $\{p_{\alpha}\}$ of a covering $\{U_{\alpha}\}_{\alpha= 1}^N$ on $X$. That is, for any $x\in X, \sum_{\alpha=1}^N p_{\alpha}(x)=1, $ and supp $p_{\alpha} \subset U_{\alpha}$.
Set
\begin{eqnarray*}
p_{\alpha}^l(x,t):=p_{\alpha}(x)\eta^l(t).
\end{eqnarray*}

We first prove that $p_{\alpha}^3 \tilde{\sigma}_t$ is $C^{k+1+\alpha}$. Let $\triangle_i^h$ be the difference quotient which is defined on Page 367 in \cite{Ko1}. We have that
\begin{eqnarray*}
\Delta_{\overline\partial}\triangle_i^h p_{\alpha}^3 \tilde{\sigma}_t&=&
(\Delta_{\overline\partial}\triangle_i^h p_{\alpha}^3 \tilde{\sigma}_t-
\triangle_i^h \Delta_{\overline\partial}  p_{\alpha}^3 \tilde{\sigma}_t)+(\triangle_i^h  \Delta_{\overline\partial}p_{\alpha}^3\tilde{\sigma}_t-
\triangle_i^h p_{\alpha}^3 \Delta_{\overline\partial}\tilde{\sigma}_t)
+\triangle_i^h p_{\alpha}^3  \Delta_{\overline\partial}\tilde{\sigma}_t\\
&=&\triangle_i^h p_{\alpha}^3(-\overline\partial^{\ast}\partial (\epsilon \tilde{\sigma}_t)+
\overline\partial \overline\partial^{\ast}  \overline\partial \overline\partial^{\ast} \partial^{\ast} G_{BC} \partial (\epsilon \tilde{\sigma}_t)+\overline\partial \overline\partial^{\ast}\tilde{\sigma}_{00})\\
&&(\Delta_{\overline\partial}\triangle_i^h p_{\alpha}^3 \tilde{\sigma}_t-
\triangle_i^h \Delta_{\overline\partial}  p_{\alpha}^3 \tilde{\sigma}_t)+(\triangle_i^h  \Delta_{\overline\partial}p_{\alpha}^3\tilde{\sigma}_t-
\triangle_i^h p_{\alpha}^3 \Delta_{\overline\partial}\tilde{\sigma}_t)\\
&:=&F^1.
\end{eqnarray*}

Since $\Delta_{\overline\partial}$ is an operator of diagonal type in the principle part which satisfies the assumption in Theorem 2.3 on Page 417 in \cite{Ko1}, we have the estimate
\begin{eqnarray*}
||\triangle_i^h p_{\alpha}^3 \tilde{\sigma}_t||_{k+\alpha}&\le&C_k(||\Delta_{\overline\partial}\triangle_i^h p_{\alpha}^3 \tilde{\sigma}_t||_{k-2+\alpha}
+||\triangle_i^h p_{\alpha}^3 \tilde{\sigma}_t||_{0})\\
&=&C_k(||F^1||_{k-2+\alpha}
+||\triangle_i^h p_{\alpha}^3 \tilde{\sigma}_t||_{0}),
\end{eqnarray*}
where $C_k$ is a positive constant which only depend on $k, \alpha$.

Now we make an estimate of $F^1$.
\begin{eqnarray*}
||F^1||_{k-2+\alpha}&\le& ||\triangle_i^h p_{\alpha}^3\overline\partial^{\ast}\partial (\epsilon \tilde{\sigma}_t)||_{k-2+\alpha}+||\triangle_i^hp_{\alpha}^3
\overline\partial \overline\partial^{\ast}  \overline\partial \overline\partial^{\ast} \partial^{\ast} G_{BC} \partial (\epsilon \tilde{\sigma}_t)||_{k-2+\alpha}\\
&&+||\triangle_i^hp_{\alpha}^3\overline\partial \overline\partial^{\ast}\tilde{\sigma}_{00}||_{k-2+\alpha}+||\mbox{lower-order terms of }~\tilde{\sigma}_t||_{k-2+\alpha}\\
\end{eqnarray*}

Applying Lemma 8.1, Lemma 8.2 on Page 455 in \cite{Ko1}, we get that there exists some positive constant $L_i, L'_i (1\le i \le 4)$  which only depend on $k, \alpha$ such that
\begin{eqnarray*}
||\triangle_i^h p_{\alpha}^3\overline\partial^{\ast}\partial (\epsilon \tilde{\sigma}_t)||_{k-2+\alpha} &\le&
L_1||p_{\alpha}^1 \epsilon ||_0||\triangle_i^h  p_{\alpha}^3 \tilde{\sigma}_t  ||_{k+\alpha} + L_1^{'}||p_{\alpha}^1\epsilon    ||_{k+\alpha}|| p_{\alpha}^3   \tilde{\sigma}_t||_{k+\alpha},
\end{eqnarray*}

\begin{eqnarray*}
||\triangle_i^hp_{\alpha}^3
\overline\partial \overline\partial^{\ast}  \overline\partial \overline\partial^{\ast} \partial^{\ast} G_{BC} \partial (\epsilon \tilde{\sigma}_t)||_{k-2+\alpha}&\le&
L_2||p_{\alpha}^1\epsilon  ||_0|| \triangle_i^h p_{\alpha}^3 \tilde{\sigma}_t  ||_{k+\alpha} + L_2^{'}||p_{\alpha}^1\epsilon     ||_{k+\alpha}|| p_{\alpha}^3   \tilde{\sigma}_t||_{k+\alpha},
\end{eqnarray*}

\begin{eqnarray*}
||\triangle_i^hp_{\alpha}^3\overline\partial \overline\partial^{\ast}\tilde{\sigma}_{00}||_{k-2+\alpha}&\le&
L_3||\tilde{\sigma}_{00}||_{k+1+\alpha},
\end{eqnarray*}

\begin{eqnarray*}
||\mbox{lower-order terms of }~\tilde{\sigma}_t||_{k-2+\alpha}&\le&
 L_4^{'}||p_{\alpha}^1\epsilon     ||_{k+\alpha}|| p_{\alpha}^3   \tilde{\sigma}_t||_{k+\alpha}.
\end{eqnarray*}

Then we have that there exists some positive constant $M_0, M_k$ which only depend on $k, \alpha$ such that
\begin{eqnarray*}
||F^1||_{k+\alpha}&\le&M_0A(r)|| \triangle_i^h p_{\alpha}^3 \tilde{\sigma}_t  ||_{k+\alpha}+M_k||p_{\alpha}^1\epsilon     ||_{k+\alpha}|| p_{\alpha}^3   \tilde{\sigma}_t||_{k+\alpha}+M_k||\tilde{\sigma}_{00}||_{k+1+\alpha}
\end{eqnarray*}

Thus,
\begin{eqnarray*}
||\triangle_i^h p_{\alpha}^3 \tilde{\sigma}_t||_{k+\alpha}&\le&
C_k M_0 A(r) || \triangle_i^h p_{\alpha}^3 \tilde{\sigma}_t  ||_{k+\alpha}+C_kM_k||p_{\alpha}^1  \epsilon   ||_{k+\alpha}|| p_{\alpha}^3   \tilde{\sigma}_t||_{k+\alpha}
+C_k || p_{\alpha}^3   \tilde{\sigma}_t||_1\\
&&+C_kM_k||\tilde{\sigma}_{00}||_{k+1+\alpha}
\end{eqnarray*}

We choose a sufficient small $r$ which only depend on $k, \alpha$  so that
\begin{eqnarray*}
M_0C_k A(r)<\frac{1}{2}
\end{eqnarray*}

holds. Then we obtain that
\begin{eqnarray*}
||\triangle_i^h p_{\alpha}^3 \tilde{\sigma}_t||_{k+\alpha}&\le&
2C_kM_k||p_{\alpha}^1 \epsilon    ||_{k+\alpha}|| p_{\alpha}^3   \tilde{\sigma}_t||_{k+\alpha}
+2C_k || p_{\alpha}^3   \tilde{\sigma}_t||_1\\
&&+2C_kM_k||\tilde{\sigma}_{00}||_{k+1+\alpha}
\end{eqnarray*}

Since $\tilde{\sigma}_t$ is $C^{k+\alpha}$ and $\tilde{\sigma}_{00}$ is $C^{\infty}$, the right side of the above formula is  bounded. So by Lemma 8.2(iii) on Page 455 in \cite{Ko1}, we have that $p_{\alpha}^3 \tilde{\sigma}_t$
is $C^{k+1+\alpha}$.

Now, we shall prove that $p_{\alpha}^5 \tilde{\sigma}_t$
is $C^{k+2+\alpha}$. Denote $D_j$ as $\frac{\partial}{\partial z^j}$ or $\frac{\partial}{\partial z^{\overline{j}}}$. We have that

\begin{eqnarray*}
\Delta_{\overline\partial}\triangle_i^h D_j p_{\alpha}^5 \tilde{\sigma}_t&=&
(\Delta_{\overline\partial}\triangle_i^h D_j p_{\alpha}^5 \tilde{\sigma}_t-\triangle_i^h \Delta_{\overline\partial}D_j p_{\alpha}^5 \tilde{\sigma}_t)
+(\triangle_i^h \Delta_{\overline\partial}D_j p_{\alpha}^5 \tilde{\sigma}_t-\triangle_i^h D_j \Delta_{\overline\partial}p_{\alpha}^5 \tilde{\sigma}_t)\\
&&+(\triangle_i^h D_j \Delta_{\overline\partial}p_{\alpha}^5 \tilde{\sigma}_t-\triangle_i^h D_j p_{\alpha}^5\Delta_{\overline\partial} \tilde{\sigma}_t)
+\triangle_i^h D_j p_{\alpha}^5\Delta_{\overline\partial} \tilde{\sigma}_t\\
&=&\triangle_i^h D_j p_{\alpha}^5(-\overline\partial^{\ast}\partial (\epsilon \tilde{\sigma}_t)+
\overline\partial \overline\partial^{\ast}  \overline\partial \overline\partial^{\ast} \partial^{\ast} G_{BC} \partial (\epsilon \tilde{\sigma}_t)+\overline\partial \overline\partial^{\ast}\tilde{\sigma}_{00})\\
&&+(\Delta_{\overline\partial}\triangle_i^h D_j p_{\alpha}^5 \tilde{\sigma}_t-\triangle_i^h \Delta_{\overline\partial}D_j p_{\alpha}^5 \tilde{\sigma}_t)
+(\triangle_i^h \Delta_{\overline\partial}D_j p_{\alpha}^5 \tilde{\sigma}_t-\triangle_i^h D_j \Delta_{\overline\partial}p_{\alpha}^5 \tilde{\sigma}_t)\\
&&+(\triangle_i^h D_j \Delta_{\overline\partial}p_{\alpha}^5 \tilde{\sigma}_t-\triangle_i^h D_j p_{\alpha}^5\Delta_{\overline\partial} \tilde{\sigma}_t)\\
&=&\triangle_i^h (-\overline\partial^{\ast}\partial D_j p_{\alpha}^5(\epsilon \tilde{\sigma}_t)+
\overline\partial \overline\partial^{\ast}  \overline\partial \overline\partial^{\ast} \partial^{\ast} G_{BC} \partial D_j p_{\alpha}^5 (\epsilon \tilde{\sigma}_t))+\triangle_i^hD_j p_{\alpha}^5\overline\partial \overline\partial^{\ast}\tilde{\sigma}_{00}\\
&&+\mbox{lower-order terms of }~\tilde{\sigma}_t\\
&:=&F^2.
\end{eqnarray*}

 Since $\Delta_{\overline\partial}$ is an operator of diagonal type in the principle part which satisfies the assumption in  Theorem 2.3 on Page 417 in \cite{Ko1}, we have the estimate
\begin{eqnarray*}
||\triangle_i^h D_j p_{\alpha}^5 \tilde{\sigma}_t||_{k+\alpha}&\le&C_k(||\Delta_{\overline\partial}\triangle_i^hD_j p_{\alpha}^5 \tilde{\sigma}_t||_{k-2+\alpha}
+||\triangle_i^hD_j p_{\alpha}^5 \tilde{\sigma}_t||_{0})\\
&=&C_k(||F^2||_{k-2+\alpha}
+||\triangle_i^h D_jp_{\alpha}^5 \tilde{\sigma}_t||_{0}),
\end{eqnarray*}
where $C_k$ is the same positive constant as that above which only depend on $k, \alpha$.

Now we make an estimate of $F^2$.
\begin{eqnarray*}
||F^2||_{k-2+\alpha}&\le& ||\triangle_i^h \overline\partial^{\ast}\partial D_jp_{\alpha}^5(\epsilon \tilde{\sigma}_t)||_{k-2+\alpha}+||\triangle_i^h
\overline\partial \overline\partial^{\ast}  \overline\partial \overline\partial^{\ast} \partial^{\ast} G_{BC} \partial D_jp_{\alpha}^5 (\epsilon \tilde{\sigma}_t)||_{k-2+\alpha}\\
&&+||\triangle_i^hD_j p_{\alpha}^5\overline\partial \overline\partial^{\ast}\tilde{\sigma}_{00}||_{k-2+\alpha}+||\mbox{lower-order terms of }~\tilde{\sigma}_t||_{k-2+\alpha}\\
\end{eqnarray*}

Applying Lemma 8.1, Lemma 8.2 on Page 455  in \cite{Ko1}, we get that there exists some positive constant $L_i, L_i' (1\le i \le 4)$ which is only depend on $k+1, \alpha$ such that
\begin{eqnarray*}
||\triangle_i^h D_j p_{\alpha}^5\overline\partial^{\ast}\partial (\epsilon \tilde{\sigma}_t)||_{k-2+\alpha} &\le&
L_1||p_{\alpha}^3 \epsilon ||_0||\triangle_i^h D_j p_{\alpha}^5 \tilde{\sigma}_t  ||_{k+\alpha} + L_1^{'}||p_{\alpha}^3\epsilon    ||_{k+1+\alpha}|| p_{\alpha}^5   \tilde{\sigma}_t||_{k+1+\alpha}
\end{eqnarray*}

\begin{eqnarray*}
||\triangle_i^h
\overline\partial \overline\partial^{\ast}  \overline\partial \overline\partial^{\ast} \partial^{\ast} G_{BC} \partial D_jp_{\alpha}^5 (\epsilon \tilde{\sigma}_t)||_{k-2+\alpha}&\le&
L_2||p_{\alpha}^3\epsilon  ||_0|| \triangle_i^h D_j p_{\alpha}^ 5\tilde{\sigma}_t  ||_{k+1+\alpha} + L_2^{'}||p_{\alpha}^3\epsilon     ||_{k+1+\alpha}|| p_{\alpha}^5   \tilde{\sigma}_t||_{k+1+\alpha}
\end{eqnarray*}

\begin{eqnarray*}
||\triangle_i^h D_j p_{\alpha}^5\overline\partial \overline\partial^{\ast}\tilde{\sigma}_{00}||_{k-2+\alpha}&\le&
L_3||\tilde{\sigma}_{00}||_{k+2+\alpha}
\end{eqnarray*}

\begin{eqnarray*}
||\mbox{lower-order terms of }~\tilde{\sigma}_t||_{k-2+\alpha}&\le&
 L_4^{'}||p_{\alpha}^3\epsilon     ||_{k+1+\alpha}|| p_{\alpha}^5   \tilde{\sigma}_t||_{k+1+\alpha}
\end{eqnarray*}

Then we have that there exists some positive constant $M_{k+1}$ which is only depend on $k+1, \alpha$ and the same $M_0$ as that above such that
\begin{eqnarray*}
||F^2||_{k+\alpha}&\le&M_0A(r)|| \triangle_i^h D_j p_{\alpha}^5 \tilde{\sigma}_t  ||_{k+\alpha}+M_{k+1}||p_{\alpha}^3\epsilon     ||_{k+1+\alpha}|| p_{\alpha}^5   \tilde{\sigma}_t||_{k+1+\alpha}+M_{k+1}||\tilde{\sigma}_{00}||_{k+2+\alpha}
\end{eqnarray*}

Thus,
\begin{eqnarray*}
||\triangle_i^h D_j p_{\alpha}^5 \tilde{\sigma}_t||_{k+\alpha}&\le&
C_k M_0 A(r) || \triangle_i^h D_j p_{\alpha}^5 \tilde{\sigma}_t  ||_{k+\alpha}+C_kM_{k+1}||p_{\alpha}^3  \epsilon   ||_{k+1+\alpha}|| p_{\alpha}^5   \tilde{\sigma}_t||_{k+1+\alpha}
\\
&&+C_k || p_{\alpha}^5   \tilde{\sigma}_t||_2+C_kM_{k+1}||\tilde{\sigma}_{00}||_{k+2+\alpha}
\end{eqnarray*}

We choose a sufficient small $r$ so that
\begin{eqnarray*}
M_0C_k A(r)<\frac{1}{2}.
\end{eqnarray*}

holds. Then we obtain that
\begin{eqnarray*}
||\triangle_i^h D_j p_{\alpha}^5 \tilde{\sigma}_t||_{k+\alpha}&\le&
2C_kM_{k+1}||p_{\alpha}^3 \epsilon    ||_{k+1+\alpha}|| p_{\alpha}^5   \tilde{\sigma}_t||_{k+1+\alpha}
+2C_k || p_{\alpha}^5   \tilde{\sigma}_t||_2\\
&&+2C_k M_{k+1}||\tilde{\sigma}_{00}||_{k+2+\alpha}
\end{eqnarray*}

Since $p_{\alpha}^3\tilde{\sigma}_t$ is $C^{k+1+\alpha}$ and $\tilde{\sigma}_{00}$ is $C^{\infty}$, the right side of the above formula is  bounded. So by Lemma 8.2(iii) on Page 455 in  \cite{Ko1}, we have that $p_{\alpha}^5 \tilde{\sigma}_t$
is $C^{k+2+\alpha}$.

Similarly, we can prove that $p_{\alpha}^{2l+1}\tilde{\sigma}_t$ is $C^{k+l+\alpha}$, for any $l\in \mathbb{N}$ where $r$ can be chosen such that it is only depend on $k, \alpha$ and is independent of $l$.
Since $p_{\alpha}^{2l+1}=1$ on $|t|\le \frac{r}{2},$ $\tilde{\sigma}_t$ is $C^{\infty}$ on $X$ with $|t|\le \frac{r}{2}.$ Then $\tilde{\sigma}_t$ can be considered as a real analytic family of forms in $t$ and thus it is smooth on $t$.

\end{proof}

Next, we discuss the locally extensions form in a special case.

\begin{theorem}

Let $(X,G)$ be a compact Hermitian generalized complex manifold.  If $H^{k+1}_{\overline\partial}(X)=0,$  we can also construct $\sigma_t$ as that in Theorem \ref{construction}, such that $\mathcal{E}(\sigma_t) \in U^k(X_{\epsilon(t)})$ and  $\overline\partial_t \circ \mathcal{E}(\sigma_t)=0$, where $\overline\partial_t$ is
the $\overline\partial$-operator on $X_{\epsilon(t)}$.
\end{theorem}

\begin{proof}
We also construct $\sigma_t$.\\

Set $\tilde{\sigma_t}:=(1-\epsilon\epsilon^{\ast})(\sigma_t).$ By Theorem \ref{k-cri},

\begin{eqnarray*}
\overline\partial_t(\mathcal{E} (\sigma_t))=0
&\Leftrightarrow& ([\partial, \epsilon\cdot]+\overline\partial)
\circ (1 - \epsilon\epsilon^{\ast})( \sigma_t)=0\\
&\Leftrightarrow& ([\partial, \epsilon\cdot]+\overline\partial)
\circ \tilde{\sigma_t}=0.
\end{eqnarray*}

Substitute
\begin{eqnarray*}
\epsilon(t):=\sum_{i+j\ge 1} t^i \bar{t}^j\epsilon_{ij},\\
\sigma_t:=\sigma_{00}+\sum_{p+q\ge 1} t^p \bar{t}^q\sigma_{pq},\\
\tilde{\sigma_t}:=\tilde{\sigma}_{00}+\sum_{p+q\ge 1} t^p \bar{t}^q\tilde{\sigma}_{pq}
\end{eqnarray*}
into the above formula and compare the coefficients of $t^p \bar{t}^q$, we get that
\begin{eqnarray}
\overline\partial \tilde{\sigma}_{pq} &=&-\sum_{i+j=p, k+l=q, i+k\ge 1}\partial (\epsilon_{ik} \cdot \tilde{\sigma}_{jl})
+\sum_{i+j=p, k+l=q, i+k\ge 1} \epsilon _{ik}\cdot \partial \tilde{\sigma}_{jl}
.~(p+q\ge 1)\label{hol03}
\end{eqnarray}

By Proposition \ref{d-closed representation}, we have that
\begin{eqnarray*}
d\tilde{\sigma}_{00}=0.
\end{eqnarray*}

Set
\begin{eqnarray*}
\tilde{\eta} _{pq}&:=&-\sum_{i+j=p, k+l=q, i+k\ge 1}\partial (\epsilon_{ik} \cdot \tilde{\sigma}_{jl}),\\
\tilde{\tau}_{pq}&:=&\tilde{\eta}_{pq}
+\sum_{i+j=p, k+l=q, i+k\ge 1} \epsilon _{ik}\cdot \partial \tilde{\sigma}_{jl}.
\end{eqnarray*}

Now we construct $\tilde{\sigma}_{pq}~ (p+q\ge 1)$ by induction on $p+q$.

Since
\begin{eqnarray*}
\overline\partial  \tilde{\tau}_{10}&=&\overline\partial (-\partial (\epsilon_{10} \cdot \tilde{\sigma}_{00})
+\epsilon _{10}\cdot \partial \tilde{\sigma}_{00})\\
&=&\partial\overline\partial  (\epsilon_{10} \cdot \tilde{\sigma}_{00})
+\overline\partial (\epsilon _{10}\cdot \partial \tilde{\sigma}_{00})\\
&=&\partial(d_L  \epsilon_{10} \cdot \tilde{\sigma}_{00}+\epsilon_{10} \cdot \overline\partial\tilde{\sigma}_{00})
+d_L \epsilon _{10}\cdot \partial \tilde{\sigma}_{00}+\epsilon _{10}\cdot\overline\partial \partial \tilde{\sigma}_{00}\\
&=&0  ~~(\mbox{since}~~d\tilde{\sigma}_{00}=d_L \epsilon _{10}=0).
\end{eqnarray*}

Then by the assumption $H^{k+1}_{\overline\partial}(X)=0,$ we have that $\tilde{\tau}_{10} \in Im \overline\partial,$ that is,
\begin{eqnarray*}
  \tilde{\tau}_{10}=\overline\partial \tilde{\sigma}_{10}
\end{eqnarray*}
has a solution $\tilde{\sigma}_{10}$.

Similarly, we can find $\tilde{\sigma}_{01}$ in the same way.

If we have already got $\tilde{\sigma}_{pq}$ satisfies Formula \ref{hol03} , where $p+q=1,2,\cdots , N-1$. For $p+q=N,$
we have that

\begin{eqnarray*}
\overline\partial \eta_{pq}
&:=&\sum_{i+j=p, k+l=q, i+k\ge 1}\partial \overline\partial(\epsilon_{ik} \cdot \tilde{\sigma}_{jl})\\
&=&\sum_{i+j=p, k+l=q, i+k\ge 1}\partial (d_L\epsilon_{ik} \cdot \tilde{\sigma}_{jl}+\epsilon_{ik} \cdot\overline\partial \tilde{\sigma}_{jl})\\
&=&\sum_{i+j=p, k+l=q, i+k\ge 1}\partial (\sum_{r+m=i, s+n=k}\frac{1}{2}[\epsilon_{rs}, \epsilon_{mn}] \cdot \tilde{\sigma}_{jl}+\epsilon_{ik} \cdot\overline\partial \tilde{\sigma}_{jl})\\
&=&\sum_{i+j=p, k+l=q, i+k\ge 1}\partial \sum_{r+m=i, s+n=k}\frac{1}{2}[\epsilon_{rs}, \epsilon_{mn}] \cdot \tilde{\sigma}_{jl}\\
&&+\sum_{i+j=p, k+l=q, i+k\ge 1}
\sum _{r+m=j, s+n=l}\partial (\epsilon_{ik} \cdot(-\partial (\epsilon_{rs}\cdot \tilde{\sigma}_{mn})+\epsilon_{rs}\cdot \partial \tilde{\sigma}_{mn}))\\
&=&\sum_{i+j=p, k+l=q, i+k\ge 1}\partial (\frac{1}{2}(-\partial (\epsilon_{rs}\cdot \epsilon_{mn} \cdot\tilde{\sigma}_{jl})-  \epsilon_{rs}\cdot \epsilon_{mn} \cdot\partial\tilde{\sigma}_{jl}
+\epsilon_{rs}\cdot \partial (\epsilon_{mn} \cdot\tilde{\sigma}_{jl})\\
&&+\epsilon_{mn}\cdot \partial (\epsilon_{rs} \cdot\tilde{\sigma}_{jl}   ))+\sum_{i+j=p, k+l=q, i+k\ge 1}\sum _{r+m=j, s+n=l}\partial( \epsilon_{ik} \cdot(-\partial( \epsilon_{rs}\cdot \tilde{\sigma}_{mn})+\epsilon_{rs}\cdot \partial \tilde{\sigma}_{mn}))\\
&=&\sum_{i+j=p, k+l=q, i+k\ge 1}\partial (-\frac{1}{2}  \epsilon_{rs}\cdot \epsilon_{mn} \cdot\partial\tilde{\sigma}_{jl} +\epsilon_{rs}\cdot \partial (\epsilon_{mn} \cdot\tilde{\sigma}_{jl})\\
&&+\sum_{i+j=p, k+l=q, i+k\ge 1}\sum _{r+m=j, s+n=l}\partial( \epsilon_{ik} \cdot(-\partial (\epsilon_{rs}\cdot \tilde{\sigma}_{mn})+\epsilon_{rs}\cdot \partial \tilde{\sigma}_{mn}))\\
&=&\frac{1}{2}\sum_{i+r+m=p, k+s+n=q, i+k\ge 1, r+e\ge 1}\partial(  \epsilon_{ik}\cdot \epsilon_{rs} \cdot\partial\tilde{\sigma}_{mn}).
\end{eqnarray*}

The third equality holds  since the integrable condition; the fourth equality holds  since the induction; the fifth equality holds  since Lemma \ref{braket};
and the sixth equality holds  since $\partial ^2=0$.

Then
\begin{eqnarray*}
\overline\partial  \tilde{\tau}_{pq}&:=&\overline\partial\tilde{\eta}_{pq}
+\overline\partial \sum_{i+j=p, k+l=q, i+k\ge 1} \epsilon _{ik}\cdot \partial \tilde{\sigma}_{jl}\\
&=&\frac{1}{2}\sum_{i+r+m=p, k+s+n=q, i+k\ge 1, r+e\ge 1}\partial(  \epsilon_{ik}\cdot \epsilon_{rs} \cdot\partial\tilde{\sigma}_{mn})
+\sum_{i+j=p, k+l=q, i+k\ge 1}(d_L  \epsilon _{ik})\cdot \partial \tilde{\sigma}_{jl}\\
&&-\sum_{i+j=p, k+l=q, i+k\ge 1}  \epsilon _{ik}\cdot \partial \overline\partial\tilde{\sigma}_{jl}\\
&=&\frac{1}{2}\sum_{i+r+m=p, k+s+n=q, i+k\ge 1, r+e\ge 1}\partial(  \epsilon_{ik}\cdot \epsilon_{rs} \cdot\partial\tilde{\sigma}_{mn})
\\
&&+\sum_{i+j=p, k+l=q, i+k\ge 1}(\frac{1}{2}\sum_{r+m=i, s+n=k}[\epsilon_{rs}, \epsilon_{mn}])\cdot \partial \tilde{\sigma}_{jl}\\
&&-\sum_{i+j=p, k+l=q, i+k\ge 1}  \epsilon _{ik}\cdot \partial \overline\partial\tilde{\sigma}_{jl}\\
&=&\frac{1}{2}\sum_{i+r+m=p, k+s+n=q, i+k\ge 1, r+e\ge 1}\partial(  \epsilon_{ik}\cdot \epsilon_{rs} \cdot\partial\tilde{\sigma}_{mn})
\\
&&+\sum_{i+j=p, k+l=q, i+k\ge 1}(\frac{1}{2}\sum_{r+m=i, s+n=k}(
-\partial (\epsilon_{mn}\cdot \epsilon_{rs} \cdot \partial \tilde{\sigma}_{jl})-\epsilon_{mn}\cdot \epsilon_{rs} \cdot \partial^2 \tilde{\sigma}_{jl}
\\
&&+\epsilon_{rs}\cdot \partial(\epsilon_{mn} \cdot  \partial\tilde{\sigma}_{jl})+\epsilon_{mn}\cdot \partial(\epsilon_{rs} \cdot  \partial\tilde{\sigma}_{jl})
)\\
&&-\sum_{i+j=p, k+l=q, i+k\ge 1}  \epsilon _{ik}\cdot \partial (\sum_{r+m=j, s+n=l, r+s\ge1, m+n\ge1}-\partial (\epsilon_{rs}\cdot \tilde{\sigma}_{mn})+\epsilon_{rs}\cdot \partial\tilde{\sigma}_{mn} )\\
&=&\frac{1}{2}\sum_{i+r+m=p, k+s+n=q, i+k\ge 1, r+e\ge 1}\partial(  \epsilon_{ik}\cdot \epsilon_{rs} \cdot\partial\tilde{\sigma}_{mn})
\\
&&+\sum_{i+j=p, k+l=q, i+k\ge 1}(\frac{1}{2}\sum_{r+m=i, s+n=k}(
-\partial (\epsilon_{mn}\cdot \epsilon_{rs} \cdot \partial \tilde{\sigma}_{jl})
\\
&&+\epsilon_{rs}\cdot \partial(\epsilon_{mn} \cdot \partial \tilde{\sigma}_{jl})+\epsilon_{mn}\cdot \partial(\epsilon_{rs} \cdot \partial \tilde{\sigma}_{jl})
)\\
&&-\sum_{i+j=p, k+l=q, i+k\ge 1}  \epsilon _{ik}\cdot \partial (\sum_{r+m=j, s+n=l, r+s\ge1, m+n\ge1}\epsilon_{rs}\cdot \partial\tilde{\sigma}_{mn} )\\
&=&0.
\end{eqnarray*}

The third equality holds  since the integrable condition; the fourth equality holds  since Lemma \ref{braket} and the induction.

Then by the assumption $H^{k+1}_{\overline\partial}(X)=0,$ we have that $\tilde{\tau}_{pq} \in Im \overline\partial,$ that is,
\begin{eqnarray*}
 \tilde{\tau}_{pq}=\overline\partial \tilde{\sigma}_{pq}
\end{eqnarray*}
has a solution $\tilde{\sigma}_{pq}$.

We get the regularity of $\tilde{\sigma}_t$ by using the same elliptic estimates as that in the above theorem.

\end{proof}

\section{Applications}
In this section, we use the extension formula in Theorem \ref{construction} to get the invariance of the generalized Hodge number of the deformations of compact
generalized Hermitian manifolds with $\partial\overline\partial$-lemma holds. The method we used here is parallel to that in \cite{Ko1, RaoWanZhao03}.
Our idea is as follows: Since $G_{\overline\partial}$ is strongly elliptic, by Theorem 7.3 on Page 326 in  \cite {Ko1}, we have that $h^{k}_{\overline\partial_t}(X_t):=dim_{\mathbb{C}}H^{k-1}_{\overline\partial _t}(X_t)$ is upper-semicontinuous in $t$, that is,

\begin{eqnarray*}
h^{k}_{\overline\partial_t}(X_t) &\le& h^{k}_{\overline\partial}(X) ~\mbox{if}~ |t| ~\mbox{sufficient small}.
\end{eqnarray*}

In Theorem \ref{construction}, we have given a map  from $H_{\overline\partial }^k (X) $ to  $H_{\overline\partial_t }^k (X_t)$
which is Map \ref{cohominj} below. If it is injective, we have  that $h^{k}_{\overline\partial_t}(X_t) \ge h^{k}_{\overline\partial}(X)$ and thus get the result.

\begin{proposition}\label{cohoinj}
Let $(X,G)$ be a compact generalized Hermitian manifold. If we assume that $h^{k-1}_{\overline\partial _t}(X_t)$ is independent of $t$, and $X \in \mathbb{B}^{k-1}\cap \mathbb{S}^{k+1},$ we have that

\begin{eqnarray}\label{cohominj}
H_{\overline\partial }^k (X)  & \to &  H_{\overline\partial_t }^k (X_t) \\
\sigma_{00} & \mapsto &  \mathcal{E}(\sigma_{t})\nonumber
\end{eqnarray}

is injective, where $\sigma_t$ is constructed in Theorem \ref{construction}.
\end{proposition}

\begin{proof}

Since $X \in \mathbb{B}^{k-1}\cap \mathbb{S}^{k+1},$ by Theorem \ref{construction}, we have that $\sigma_{t}$ exists.
If there exists some $\eta_t\in U^{k-1}(X_t) $ such that $\mathcal{E}(\sigma_t)=\overline\partial _t \eta_t$, then
\begin{eqnarray*}
\mathcal{E}(\sigma_t)&=&\overline\partial _t \eta_t\\
&=&\overline\partial _t (\mathbb{H}_t \eta_t+ \Delta_{\overline\partial _t } G_{\overline\partial _t }\eta_t)\\
&=&\overline\partial _t  \overline\partial _t^{\ast} \overline\partial _t G_{\overline\partial _t }\eta_t\\
&=&\overline\partial _t G_{\overline\partial _t } \overline\partial _t^{\ast} \overline\partial _t \eta_t\\
&=&\overline\partial _t G_{\overline\partial _t } \overline\partial _t^{\ast} (\mathcal{E}(\sigma_t)).
\end{eqnarray*}

Also, by the assumption that $h^{k-1}_{\overline\partial _t}(X_t)$ is independent of $t$, we have $G_{\overline\partial _t }$
is smooth on $t$~(Theorem 7.6 on Page 344 in \cite{Ko1}). So we can let $t \to 0$ on both sides of the formula above and  get that
\begin{eqnarray*}
\sigma_{00}=\overline\partial G_{\overline\partial  } \overline\partial ^{\ast} (\sigma_{00}),
\end{eqnarray*}
that is, $\sigma_{00}$ is also $\overline\partial$-exact.
\end{proof}

\begin{proposition}\label{case-n}
Let $(X,G)$ be a compact generalized Hermitian manifold. If we assume that $X\in \mathbb{S}^{-n+1}$, then $h^{-n}_{\overline\partial_t}(X_t)$ is independent of $t$, where $n=dim _{\mathbb{C}} X$.
\end{proposition}

\begin{proof}
For any  $\sigma_{00} \in H_{\overline\partial }^{-n} (X)$, we have that $d \sigma_{00}=\partial\sigma_{00}=0.$
By the assumption that $X\in \mathbb{S}^{-n+1}
$  and  $\partial (\sigma)=0$ for any $\sigma \in U^{-n}(X),$ we have that Equation \ref{hol03} has the solution $\tilde{\sigma}_{pq}$ and thus $\tilde{\sigma}_t$ exists.

If $\mathcal{E}(\tilde{\sigma}_t)=0 \in H_{\overline\partial_t }^{-n} (X_t)$, we have that there exists some $\eta_t \in U^{-n-1}(X_t)$, such that $\mathcal{E}(\tilde{\sigma}_t)=\overline\partial_t \eta_t$. Since $U^{-n-1}(X_t)=0$, we have that $\mathcal{E}(\tilde{\sigma}_t)=0$. By Proposition \ref{isoeE}, $\mathcal{E}$ is an isomorphism and thus $\sigma_t=0.$ Then we get that $\sigma_{00}=0$ and Map \ref{cohominj} is injective.

And since we have already known that $h^{k}_{\overline\partial_t}(X_t)$ is upper-semicontinuous in $t$ (Theorem 7.3 on Page 326 in \cite {Ko1} and $\Delta_{\overline\partial}$ is strongly elliptic), we prove the proposition.
\end{proof}

Thus, we can get the following:

\begin{corollary}\label{i}
Let $(X,G)$ be a compact generalized Hermitian manifold which satisfies $\partial\overline\partial$-lemma, then $h_{\overline\partial_t}^k(X_t)$ is independent of $t$, where $-n\le k \le n$.
\end{corollary}

\begin{proof}
Since $X$ satisfies $\partial\overline\partial$-lemma, by Lemma \ref{dd}, we have that $X\in \mathbb{B}^k $ for any $-n\le k \le n$.
We prove $h_{\overline\partial_t}^k(X_t)$ is independent of $t$ by induction on $t$.
By Proposition \ref{case-n}, we know that $h_{\overline\partial_t}^{-n}(X_t)$ is independent of $t$.
By Theorem \ref{construction}, we give a map from $H_{\overline\partial }^k (X)   \to   H_{\overline\partial_t }^k (X_t)$.
If $h_{\overline\partial_t}^{k-1}(X_t)$ is independent of $t$, by Proposition \ref{cohoinj}, we know that
this map  is injective.
Combine  the fact that $h^{k}_{\overline\partial_t}(X_t)$ is upper-semicontinuous in $t$, we can get the result that $h_{\overline\partial_t}^{k}(X_t)$ is independent of $t$.

\end{proof}

\bigskip

\noindent Kang Wei\\
Center of Mathematical Sciences, \\Zhejiang University, \\Hangzhou, Zhejiang 310027,\\
China.\\
E-mail: kangkangspbr@163.com, kang\_wei5@hotmail.com.

\end{document}